\documentclass[a4paper]{amsart}
\usepackage[latin1]{inputenc}
\usepackage{ae,aecompl,amsbsy,amssymb,amsmath,amsthm,
eurosym,amsfonts,epsfig,graphicx,graphics,verbatim,enumerate,esint,MnSymbol}
\usepackage[usenames, dvipsnames]{color}

\newcommand{\quotes}[1]{\lq#1\rq}

\newcommand{\norm}[1]{\lVert #1 \rVert}

\newcommand{\br}{\mathbb{R}}

\newcommand{\bz}{\mathbb{Z}}

\newcommand{\cd}{\mathcal{D}}

\newcommand{\cf}{\mathcal{F}}

\newcommand{\abs}[1]{\lvert#1\rvert}

\newcommand{\dsigma}{\mathrm{d}\sigma}
\newcommand{\domega}{\mathrm{d}\omega}
\newcommand{\dmu}{\mathrm{d}\mu}

\newcommand{\dt}{\mathrm{d}t}

\newcommand{\cdotroomy}{\,\cdot\,}

\newcommand{\cq}{\mathcal{Q}}

\newcommand{\angles}[1]{\langle #1 \rangle}

\usepackage{enumitem}

\newcommand{\conditionaa}{\int \big( \sum_Q \lambda_Q \frac{\omega(Q)}{\sigma(Q)} a_Q^{-1} 1_Q \big)^{p'} \dsigma < \infty} 
\newcommand{\conditionab}{\int \big( \sup_Q a_Q 1_Q \big)^{\frac{q}{1-q}} \domega < \infty}
\newcommand{\conditionba}{\sup_Q \frac{1}{\sigma(Q)}\sum_{R\subseteq Q} \lambda_R b_R^{-1} \omega(R) < \infty}
\newcommand{\conditionbb}{\sum_Q \lambda_Q a_Q^{-p'} b_Q^{p'-1} \omega(Q) <\infty}
\newcommand{\conditionbc}{\int \big( \sup_Q a_Q 1_Q \big)^{\frac{q}{1-q}} \domega < \infty}
\newcommand{\conditionda}{\sup_Q \frac{1}{\sigma(Q)} \sum_{R\subseteq Q} \lambda_Q a_Q^{-1} \omega(Q)  < \infty}
\newcommand{\conditiondb}{\int \big( \sum_Q \lambda_Q a_Q^{p'-1} 1_Q \big)^{\frac{(p-1)q}{p-q}}  \domega < \infty}
 
 \newcommand{\conditionca}{\sup_Q \frac{b_Q^{-1}}{a_Q^{-1}\sigma(Q)} \sum_{R\subseteq Q} \lambda_R a_R^{-1} \omega(R)< \infty}
 \newcommand{\conditioncb}{\sum_Q \lambda_Q a_Q^{-p'} b_Q^{p'-1} \omega(Q) < \infty }\newcommand{\conditioncc}{\int \big( \sup_Q a_Q 1_Q \big)^{\frac{q}{1-q}} \domega < \infty}

\newcommand{\conditionea}{\sup_Q \frac{1}{\sigma(Q)} \sum_{R\subseteq Q} \lambda_Q a_Q^{-1} \omega(Q)  < \infty }
\newcommand{\conditioneb}{\sum_Q \lambda_Q b_Q^{-p'} \omega(Q) a_Q^{p'-1} < \infty }
\newcommand{\conditionec}{\int \big( \sup b_Q 1_Q \big)^{\frac{q}{1-q}} \domega <\infty}

\theoremstyle{plain}
\newtheorem{theorem}{Theorem}[section]
\newtheorem{lemma}[theorem]{Lemma}
\newtheorem{proposition}[theorem]{Proposition}

\theoremstyle{definition}
\newtheorem{definition}[theorem]{Definition}
\newtheorem{problem}[theorem]{Problem}

\theoremstyle{remark}
\newtheorem*{remark}{Remark}

\numberwithin{equation}{section}

\begin{document}
\title[Two-weight $L^p\to L^q$ bounds for  positive operators]{Two-weight $L^p\to L^q$ bounds for positive dyadic operators in the case $0<q< 1 \le  p<\infty$}

\author{Timo S. H\"anninen}
\address{Department of Mathematics and Statistics, University of Helsinki, P.O. Box 68, FI-00014 HELSINKI, FINLAND}
\email{timo.s.hanninen@helsinki.fi}
\author{Igor E. Verbitsky}
\address{Department of Mathematics, University of Missouri, Columbia, MO  65211, USA}
\email{verbitskyi@missouri.edu}

\thanks{T.S.H. is supported by the Academy of Finland through funding of his postdoctoral researcher post (Funding Decision No 297929), and by the Jenny and Antti Wihuri Foundation through covering of expenses of his visit to the University of Missouri. He is a member of the Finnish Centre of Excellence in Analysis and Dynamics Research.}

\keywords{Two-weight norm inequalities, positive dyadic potential operators, Wolff potentials, discrete Littlewood--Paley spaces}
\subjclass[2010]{42B25, 42B35, 47G40}
\begin{abstract}

Let $\sigma$, $\omega$ be measures on $\br^d$, and let $\{\lambda_Q\}_{Q\in\mathcal{D}}$ be 
a family of non-negative reals indexed by the collection $\cd$ of dyadic cubes in $\br^d$. We characterize the two-weight norm inequality, 
\begin{equation*}
\lVert T_\lambda(f\sigma)\rVert_{L^q(\omega)}\le C \,  \lVert f \rVert_{L^p(\sigma)}\quad \text{for every $f\in L^p(\sigma)$,}
\end{equation*}
for the positive dyadic operator 
\begin{equation*}
T_\lambda(f\sigma):= \sum_{Q\in \mathcal{D}} \lambda_Q \, \Big(\frac{1}{\sigma(Q)} \int_Q f\mathrm{d}\sigma\Big) \, 1_Q
\end{equation*}
in the difficult range $0<q<1 \le p<\infty$ of integrability exponents. 
This range of the exponents $p, q$ appeared recently 
 in applications to nonlinear PDE, 
 which was one of the motivations for our study.

Furthermore, we introduce a scale of discrete Wolff potential conditions that depends monotonically on an integrability parameter, and prove that such conditions are necessary (but not sufficient)  for small parameters, and sufficient (but not necessary) for large parameters. 

Our characterization applies to Riesz potentials $I_\alpha (f \sigma) = (-\Delta)^{-\frac{\alpha}{2}} (f\sigma) $ ($0<\alpha<d$), since it is known that they  can be controlled by  model dyadic operators. The weighted norm inequality for Riesz potentials 
in this range  of $p, q$ 
has been characterized previously only in the special case where $\sigma$ is Lebesgue measure.

\end{abstract}
\maketitle
\tableofcontents
%
%
\section*{Notation}
\begin{tabular}{c l}
$\cd$ & The collection of all the dyadic cubes.\\
$\cq$ & The collection of all the dyadic cubes  \\
& that contribute to the  two-weight norm inequality \eqref{eq_norminequality}, \\
&$\cq:=\{Q\in \cd : \lambda_Q>0, \sigma(Q)>0, \text{ and }  \omega(Q)>0\}.$\\
$L^p(\mu)$ & The Lebesgue space, equipped with the norm \\
& $\norm{f}_{L^p(\mu)}:=(\int \abs{f}^p \dmu)^{\frac{1}{p}}.$\\
$f^{p,q}(\mu)$ & The discrete Littlewood--Paley space, equipped with the norm\\
&  $\norm{a}_{f^{p,q}(\mu)}:=\big(\int \big( \sum_Q \abs{a_Q}^q 1_Q\big)^{\frac{p}{q}} \dmu\big)^{\frac{1}{p}},  \text{ when } 0<p<\infty, \, 0<q\le \infty,$\\
&  $\norm{a}_{f^{\infty, q}(\mu)}:=  \Big( \sup_Q  \frac{1}{\mu(Q)}\sum_{R\subseteq Q} \abs{a_R}^q \, \mu(R) \Big)^{\frac{1}{q}},\text{when } p=\infty, \, 0<q<\infty$.\\
$\angles{f}^\mu_Q$ & The average of a function $f$ on a cube $Q$,\\
& $\angles{f}^\mu_Q:=\frac{1}{\mu(Q)} \int_Q f \dmu.$\\
$p'$& The H\"older conjugate $p'\in[1,\infty]$ of an exponent $p\in[1,\infty]$,\\
& $p':=\frac{p}{p-1}$.\\
$T_\lambda(\cdotroomy\sigma)$ & The dyadic operator, \\ 
&$T_\lambda(f\sigma):=\sum_{Q \in \cq} \lambda_Q \, \angles{f}^\sigma_Q \, 1_Q.$\\
$\rho_Q$& The localized sum of the operator's coefficients,  \\
&$\rho_Q:= \sum_{R\subseteq Q}\lambda_R 1_R.$\\
$\Lambda_{\gamma,Q}$ & The $\gamma$-average of the localized sum of the operator's coefficients, \\
&$\Lambda_{\gamma,Q}:=\Big(\frac{1}{\omega(Q)}\int_Q \rho_Q^\gamma\domega \Big)^{\frac{1}{\gamma}}, \, \gamma \in \br\setminus\{0\}.
$\\
$a^{-1}$& For a family $a:=\{a_Q\}$, the family $a^{-1}$ is defined by $a^{-1}:=\{a_Q^{-1}\}$. 
\end{tabular}
\smallskip

The least constant in the $L^p(\sigma)\to L^q(\omega)$ two-weight norm inequality for the operator $T_\lambda(\cdotroomy \sigma)$ is denoted by $\norm{T_\lambda(\cdotroomy \sigma)}_{L^p(\sigma)\to L^q(\omega)}$.

The uppercase Latin letters  $P,Q,R,S$ are reserved for dyadic cubes. When the collection of the cubes is clear from the context, the indexations \quotes{$Q\in \cd$} and \quotes{$Q\in \cq$} are both abbreviated as \quotes{$Q$} in the indexation of summations, and omitted in the indexation of families (and similarly for the cubes $P,R,S$). 

The lowercase Latin letters  $a,b,c,d$ are reserved for various families $a:=\{a_Q\}$, $b:=\{b_Q\}, \ldots$ of non-negative reals, and $\lambda:=\{\lambda_Q\}$ for the fixed family of non-negative reals associated with the operator $T_\lambda(\cdotroomy \sigma)$.

Throughout this paper, we follow the usual convention $\frac{0}{0}:=0$. 

\section{Introduction}
Let $L^p(\sigma)$ and $L^q(\omega)$ denote the Lebesgue spaces associated with exponents $0< p, q \le \infty$ and locally finite Borel measures $\sigma$ and $\omega$ on $\br^d$. Let $\lambda:=\{\lambda_Q\}_{Q\in \cd}$ be a family of non-negative reals indexed by the collection  $\cd$ of dyadic cubes.  The positive dyadic operator $T_\lambda(\cdotroomy \sigma)$ associated with the coefficients $\lambda$  is defined by setting
\begin{equation}\label{eq_definition_operator}
T_\lambda(f\sigma):= \sum_{Q\in \cd} \lambda_Q \, \Big(\frac{1}{\sigma(Q)} \int_Q f\dsigma\Big) \, 1_Q
\end{equation}
for every measurable function $f:\br^d \to \br$. 

We characterize  the two-weight norm inequality
\begin{equation}\label{eq_norminequality}
\norm{T_\lambda(f\sigma)}_{L^q(\omega)}\lesssim \norm{f}_{L^p(\sigma)}\quad \text{for every $f\in L^p(\sigma)$}
\end{equation}
in the case $0<q<1 \le  p<\infty$. This range of the exponents $p, q$ appeared recently 
 in applications to nonlinear PDE \cite{cao2017,quinn2017,verbitsky2017}, 
 which was one of the motivations for our study.

Our characterization is obtained by means of factorization of the operator's coefficients $\{\lambda_Q\}$ in discrete Littlewood--Paley spaces. Furthermore, we introduce a scale of discrete Wolff potential conditions that depends monotonically on an integrability parameter $\gamma$, and prove that this potential condition is necessary (but not sufficient) for small parameters, and sufficient (but not necessary) for large parameters. 

We note that, for the two-weight norm inequality \eqref{eq_norminequality}, we may in the operator's definition \eqref{eq_definition_operator} restrict the summation over the collection  $\cd$ of all the dyadic cubes to the summation over the subcollection $\cq$ defined by
$$
\cq:=\{Q\in \cd : \lambda_Q>0, \sigma(Q)>0, \text{ and }  \omega(Q)>0\}, 
$$ because only such cubes contribute to the norm. This restriction allows us  to avoid division by zero in certain conditions.

In the case $1<p\leq q<\infty$, the two-weight inequality can be characterized by 
 a pair of the usual testing conditions  associated with inequality
\eqref{eq_norminequality} and its dual inequality  \cite{nazarov1999,lacey2009,treil2015,hytonen2014}. 

In the case $1<q<p<\infty$, the two-weight inequality can be characterized by a pair of the Wolff potential conditions \cite{cascante2004,tanaka2014}, or, alternatively, by maximal inequalities that will be introduced in the authors' forthcoming preprint.  These characterizations   have also been extended to vector-valued \cite{scurry2013,hanninen2017,lai2015} and multilinear \cite{tanaka2015,hanninen2016,tanaka2016} settings.

In contrast to these well-understood cases, the case $0<q< 1 <  p<\infty$ is less studied. To the best of the authors' knowledge, the only earlier characterizations of the two-weight norm inequality are the following:
\begin{itemize}
\item Characterization by means of a single Wolff potential condition \eqref{eq_wolff_potential_condition} under an additional restriction that the operator's coefficients satisfy the so-called {\it dyadic  logarithmic bounded oscillation} (DLBO) condition, by Cascante, Ortega, and Verbitsky \cite[Theorem A]{cascante2006}. The Wolff potential condition and the DLBO restriction are explained later in the introduction.

\item  A dual reformulation of the weighted norm inequality through rephrasing  it in terms of discrete Littlewood--Paley norms, by Cascante and Ortega \cite[Theorem 1.1]{cascante2017}. The dual reformulation is explained in Section \ref{sec_straightforward_conditions}.

\end{itemize}

\begin{remark} In the much simpler ``unweighted" case, when  $\mu=\omega=\sigma$, the inequality is characterized by the simple condition 
$\int \big(\sum_{Q\in \cd} \lambda_Q \, 1_Q \big)^{\frac{pq}{p-q}} \dmu < \infty$ (see \cite{verbitsky2017}, where more general integral operators are treated). 
\end{remark}

In this paper, we study the difficult case $0<q<1  \le   p<\infty$,  
 for general measures $\sigma$ and $\omega$, and general coefficients $\lambda$.   In the endpoint case $p=1$, some of our results hold for a related multiplier problem for discrete Littlewood--Paley 
spaces, which can be viewed as a modification of the weighted norm inequality \eqref{eq_norminequality} (see the remark  after Theorem~\ref{theorem_factorization}).

 We characterize the two-weight norm inequality \eqref{eq_norminequality} by means of the existence of an auxiliary family of coefficients satisfying a $\sigma$-Carleson and an $\omega$-integrability condition:

\begin{theorem}[Characterization via auxiliary coefficients]\label{theorem_maintheorem}Let $0<q<1 < p<\infty$. Let $\sigma$ and $\omega$ be locally finite Borel measures on $\br^d$. Let $\{\lambda_Q\}_{Q\in \cq}$ be a family of non-negative reals  associated with  the operator $T_\lambda(\cdotroomy \sigma)$. Then the following assertions hold:
\begin{enumerate}[label=(\roman*)]
\item (Sufficiency) Every family $\{a_Q\}_{Q\in\cq}$ of positive reals satisfies the estimate
\begin{equation}\label{eq_upperb}
\begin{split}
&\norm{T_\lambda(\cdotroomy \sigma)}_{L^p(\sigma)\to L^q(\omega)}\lesssim_{p,q}\\
& 
\Big(\sup_{Q\in\cq} \frac{1}{\sigma(Q)} \sum_{\substack{R\in \cq :\\R\subseteq Q}} a_R^{-1} \sigma(R) \Big)^{\frac{1}{p}}\Big(\int \Big(  \sum_{Q\in \cq} \lambda_Q^{p'}  \Big(\frac{\omega(Q)}{\sigma(Q)}\Big)^{p'-1}
a_Q^{p'-1} 1_Q \Big)^{\frac{(p-1)q}{p-q}} \domega\Big)^{\frac{p-q}{p q}}.
\end{split}
\end{equation}

\item (Necessity) There exists a family $\{a_Q\}_{Q\in\cq}$ of positive reals that satisfies the reverse of estimate \eqref{eq_upperb}.
\end{enumerate}
\end{theorem}

\begin{remark} Clearly, under the assumptions of Theorem \ref{theorem_maintheorem}, we have 
\begin{equation*}\label{eq_upperb-b}
\begin{split}
&\norm{T_\lambda(\cdotroomy \sigma)}_{L^p(\sigma)\to L^q(\omega)}\eqsim_{p,q} \, \inf_{\substack{\{a_Q\}_{Q\in\cq}:\\ a_Q>0}} \Big(\sup_{Q\in\cq} \frac{1}{\sigma(Q)} \sum_{\substack{R\in \cq :\\R\subseteq Q}} a_R^{-1} \sigma(R) \Big)^{\frac{1}{p}}\\ 
& 
\times \Big(\int \Big(  \sum_{Q\in \cq} \lambda_Q^{p'}  \Big(\frac{\omega(Q)}{\sigma(Q)}\Big)^{p'-1}
a_Q^{p'-1} 1_Q \Big)^{\frac{(p-1)q}{p-q}} \domega\Big)^{\frac{p-q}{p q}}.
\end{split}
\end{equation*}
\end{remark}

It is useful to have more concrete sufficient and necessary conditions for the two-weight norm inequality \eqref{eq_norminequality}, in order to be able to verify in practice whether such a condition holds or fails for a particular operator and particular measures. 
 By constructing specific auxiliary coefficients, we apply the characterization via auxiliary families (Theorem \ref{theorem_maintheorem}, Theorem \ref{theorem_factorization}, and their variants) to prove that the single Wolff potential condition \eqref{eq_wolff_potential_condition} is sufficient, even without imposing the DLBO restriction on the operator's coefficients. Notice that the crucial exponent $\frac{(p-1)q}{p-q}$ in condition \eqref{eq_upperb} is the same  as in the Wolff potential condition \eqref{eq_wolff_potential_condition}.

The starting point for us is the following factorization theorem, which allows us to factorize the operator's coefficients  as $\{\lambda_Q\}_{Q\in \cq}=\{b_Q c_Q \}_{Q\in \cq}$ with  non-negative coefficients $\{b_Q \}_{Q\in \cq}$ 
and $\{c_Q \}_{Q\in \cq}$ each satisfying certain integral conditions. 
We obtain this factorization by applying Maurey's  theorem \cite[Theorem 2]{maurey1974}, discretizing the factorizing density function, and using duality:

\begin{theorem}[Characterization via factorization] \label{theorem_factorization} Let $0<q< 1\le   p<\infty$. Let $\sigma$ and $\omega$ be locally finite Borel measures. Let $\{\lambda_Q\}_{Q\in \cq}$ be a family of non-negative reals, with which the operator $T_\lambda(\cdotroomy \sigma)$ is associated. Then the following assertions hold:

\begin{enumerate}[label=(\roman*)]
\item (Sufficiency) Every factorization $\{\lambda_Q\}_{Q\in \cq}=\{b_Q c_Q\}_{Q\in \cq}$ satisfies the estimate
\begin{equation}
\label{eq_factorization_suffiency}
\begin{split}
&\norm{T_\lambda(\cdotroomy \sigma)}_{L^p(\sigma)\to L^q(\omega)} \\
&\lesssim_{p,q} \Big( \int \big(\sup_{Q\in \cq} b_Q 1_Q \big)^{\frac{q}{1-q}}\domega\Big)^{\frac{1-q}{q}}  \Big(  \int \big( \sum_{Q\in \cq} c_Q \frac{\omega(Q)}{\sigma(Q)} 1_Q \big)^{p'} \dsigma \Big)^{\frac{1}{p'}}.
\end{split}
\end{equation}
\item  (Necessity) There exists a factorization  $\{\lambda_Q\}_{Q\in \cq}=\{b_Q c_Q \}_{Q\in \cq}$ that satisfies the reverse of estimate \eqref{eq_factorization_suffiency}.
\end{enumerate}
\end{theorem}
In the endpoint case $p=1$,  $0<q<1$, the statement of the theorem 
is interpreted  in a natural way:
\begin{equation}
\label{eq_factorization_1}
\begin{split}
&\norm{T_\lambda(\cdotroomy \sigma)}_{L^1(\sigma)\to L^q(\omega)} \\
&\eqsim_{q}\inf_{\substack{\{b_Q\},\{c_Q\} :\\ \{\lambda_Q\}=\{b_Q c_Q \}}} \Big( \int \big(\sup_{Q\in \cq} b_Q 1_Q \big)^{\frac{q}{1-q}}\domega\Big)^{\frac{1-q}{q}}  \Big \Vert \sum_{Q\in \cq} c_Q \, \frac{\omega(Q)}{\sigma(Q)} \, 1_Q  \Big \Vert_{L^\infty(\sigma)}.
\end{split}
\end{equation}

\begin{remark}
In the endpoint case, a different characterization was obtained by Quinn and Verbitsky \cite{quinn2017}: The two-weight inequality \eqref{eq_norminequality} holds for 
$p=1$, $0<q<1$ if and only if  there exists $u \in L^q(\sigma)$, with $u>0$ $\sigma$-almost everywhere, such that 
$T_\lambda(u^q \sigma) \le c \, u$. This characterization works also for more general integral operators. It is related to a sublinear version of Schur's test, and is motivated by applications to sublinear elliptic PDE.  
\end{remark}

We next recall the definition of {\it discrete Littlewood--Paley spaces} ${f^{r,s}(\mu)}$ for exponents 
$0<r\le \infty$, $0<s\le \infty$, and  a locally finite Borel measure $\mu$ on $\br^d$. 
This scale of spaces  was introduced originally by Frazier and Jawerth \cite{frazier1990} 
in the special case where $\mu$ is Lebesgue measure. (In fact, they introduced the closely related space $f^{r,s}_\alpha$ as the discrete space of coefficients in wavelet-type decompositions  of Lizorkin--Triebel spaces $F^{r,s}_\alpha$ via the so-called $\varphi$-transform.)

The discrete Littlewood--Paley norm $\norm{\cdotroomy}_{f^{r,s}(\mu)}$ of a family $\{c_Q\}_{Q \in \cd}$ 
of reals is defined as follows: for exponents $r\in (0,\infty)$ and $s \in (0, \infty]$,  as the mixed $L^r(l^s)$ Lebesgue norm
$$
\norm{c}_{f^{r,s}(\mu)}:= \Big( \int \big(\sum_{Q} |c_Q|^s 1_Q \big)^{\frac{r}{s}} \dmu \Big)^{\frac{1}{r}},
$$
for exponents $r=\infty$ and $s\in(0,\infty)$, as the Carleson norm
$$
\norm{c}_{f^{\infty,s}(\mu)}:= \Big( \sup_Q  \frac{1}{\mu(Q)}\sum_{R\subseteq Q} |c_R|^s \mu(R) \Big)^{\frac{1}{s}};
$$
for other exponents $r, s$, the definition is deferred to Section \ref{sec:preliminaries_littlewoodpaleynorms}.

In terms of discrete Littlewood--Paley norms, the characterization by factorization (Theorem \ref{theorem_factorization}) of the two-weight norm inequality \eqref{eq_norminequality}  for $p\in(1,\infty)$ can be expressed as
\begin{equation*}\label{eq:characterization_factorization_inf}
\norm{T_\lambda(\cdotroomy \sigma)}_{L^p(\sigma)\to L^q(\omega)}\eqsim_{p,q}\inf_{\substack{\{a_Q\},\{b_Q\} :\\ \{\lambda_Q\}=\{a_Q b_Q \}}} \norm{\{a_Q\}_{Q\in\cq}}_{f^{\frac{q}{1-q},\infty}(\omega)}\norm{ \{b_Q \frac{\omega(Q)}{\sigma(Q)}\}_{Q\in\cq}}_{f^{p',1}(\sigma)}.
\end{equation*}

\begin{remark} We observe that:
\begin{enumerate}[label=(\alph*)] 
\item For $p\in(1,\infty)$, the two-weight norm inequality \eqref{eq_norminequality} is equivalent to the multiplier inequality from $f^{p, \infty}(\sigma)$ to $f^{q, 1}(\omega)$:
\begin{equation*}
\label{eq_norminequality_sup}
\norm{T_\lambda(\cdotroomy \sigma)}_{L^p(\sigma)\to L^q(\omega)}\eqsim_{q,p} \sup_{\{d_Q\}} \frac{\norm{\sum_Q \lambda_Q d_Q 1_Q }_{L^q(\omega)}}{\norm{\sup_Q d_Q 1_Q}_{L^p(\sigma)}},
\end{equation*}
as explained in Section \ref{sec_straightforward_conditions}. 
\item 
 In the endpoint case $p=1$,  a modification of  Theorem \ref{theorem_maintheorem} and Theorem \ref{theorem_factorization} 
 for the multiplier inequality from $f^{1, \infty}(\sigma)$ to $f^{q, 1}(\omega)$ 
 states that
$$
\sup_{\{d_Q\}} \frac{\norm{\sum_Q \lambda_Q d_Q 1_Q }_{L^q(\omega)}}{\norm{\sup_Q d_Q 1_Q}_{L^1(\sigma)}}\eqsim_{q}\inf_{\substack{\{b_Q\},\{c_Q\} :\\ \{\lambda_Q\}=\{b_Q c_Q \}}} \norm{\{b_Q\}_{Q\in\cq}}_{f^{\frac{q}{1-q},\infty}(\omega)}\norm{ \{c_Q \frac{\omega(Q)}{\sigma(Q)}\}}_{f^{\infty,1}(\sigma)}.
$$
Note that
\begin{equation*}\label{eq_norm-max}
\sup_{f\in L^1(\sigma)}\frac{\norm{T_\lambda(f\sigma)}_{L^q(\omega)}}{\norm{M^\sigma f}_{L^1(\sigma)}}\leq \sup_{\{d_Q\}} \frac{\norm{\sum_Q \lambda_Q d_Q 1_Q }_{L^q(\omega)}}{\norm{\sup_Q d_Q 1_Q}_{L^1(\sigma)}},
\end{equation*}
where $M^\sigma f:= \sup_Q \angles{f}^\sigma_Q 1_Q$ is the dyadic Hardy--Littlewood maximal operator with respect to the measure $\sigma$.
\item Factorization results in Theorems  \ref{theorem_maintheorem}  and \ref{theorem_factorization} 
can be restated equivalently via interpolation theory between various multipliers for discrete Littlewood--Paley 
spaces and Carleson measures. 
\end{enumerate}
\end{remark}

Next, we consider the sufficiency and necessity of Wolff potential-type conditions. 
Let $1<p<\infty$, and let $\sigma, \omega$ be  locally finite Borel measures on $\br^d$.  
The {\it discrete Wolff potential} $W^{p'}_{\lambda,\sigma}[\omega]$ is defined  by
\begin{equation}\label{eq:def:wolfpotential}
W^{p'}_{\lambda,\sigma}[\omega]:= \sum_Q 1_Q \lambda_Q \Big(\frac{\omega(Q)}{\sigma(Q)}\Big)^{p'-1} \Big(\frac{1}{\omega(Q)} \int \big( \sum_{R:R\subseteq Q} \lambda_R 1_R \big) \domega\Big)^{p'-1}.
\end{equation}
The {\it Wolff potential condition} reads
\begin{equation}\label{eq_wolff_potential_condition}
\int \big(W^{p'}_{\lambda,\sigma}[\omega]\big)^{\frac{(p-1)q}{p-q}} \domega < \infty.
\end{equation}

The discrete two-weight version  \eqref{eq:def:wolfpotential} of the Wolff potential and the   Wolff potential condition \eqref{eq_wolff_potential_condition} 
were introduced by Cascante, Ortega, and Verbitsky \cite{cascante2006} in relation to the so-called Wolff inequality. (See \cite{adams1996} and \cite{kuusi2014} for a discussion of Wolff potentials' history and  applications in harmonic analysis, function spaces, and PDE.)

The discrete Wolff potential can be generalized as follows. For a local integrability parameter $\gamma\in\br\setminus\{0\}$, the local $\gamma$-average $\Lambda_{\gamma,Q}$ of the operator's coefficients is defined by
$$
\Lambda_{\gamma,Q}:= \Big(\frac{1}{\omega(Q)}\int_Q \big(\sum_{R:R\subseteq Q} \lambda_R 1_R \big)^\gamma\domega \Big)^{\frac{1}{\gamma}}.
$$

Using this notation, the discrete Wolff potential 
can be extended to the {\it generalized discrete Wolff potential} $W^{p'}_{\gamma, \lambda,\sigma}[\omega]$ defined by
$$
W^{p'}_{\gamma, \lambda,\sigma}[\omega]:= \sum_Q 1_Q \lambda_Q \Big(\frac{\omega(Q)}{\sigma(Q)}\Big)^{p'-1} \Lambda_{\gamma,Q}^{p'-1},
$$
and, accordingly, the {\it generalized Wolff potential condition} reads
\begin{equation}\label{eq_generalized_wolff_potential_condition}
\int \big(W^{p'}_{\gamma,\lambda,\sigma}[\omega]\big)^{\frac{(p-1)q}{p-q}} \domega < \infty.
\end{equation}
The generalized discrete Wolff potential $W^{p'}_{\gamma,\lambda,\sigma}[\omega]$ is increasing in the integrability parameter $\gamma$, by Jensen's inequality, and coincides with the usual discrete Wolff potential when $\gamma=1$. 

In the case $q\in (1,\infty)$, the two-weight norm inequality \eqref{eq_norminequality} is characterized by a pair of the Wolff potential conditions: the Wolff potential condition \eqref{eq_wolff_potential_condition}  together with its dual counterpart are necessary \cite[Theorem B's first assertion]{cascante2004}, and  sufficient \cite[Theorem 1.3]{tanaka2014} 
for \eqref{eq_norminequality}. In the borderline case $q=1$, $1<p<\infty$, the Wolff potential condition \eqref{eq_wolff_potential_condition} alone is both necessary and 
sufficient \cite{cascante2006}.

By contrast, in the case $q\in(0,1)$, no characterization by Wolff potential conditions is known in the general case.  In the special case that the operator's coefficients satisfy the {\it dyadic logarithmic bounded oscillation } (DLBO) condition,
\begin{equation}\label{eq:dlbo}
\sup_{x\in Q} \big( \sum_{R\subseteq Q} \lambda_R 1_R(x) \big)\lesssim \inf_{x\in Q} \big( \sum_{R\subseteq Q} \lambda_R 1_R(x) \big) \quad\text{for every $Q$},
\end{equation}
the two-weight norm inequality is characterized by the Wolff potential condition \eqref{eq_wolff_potential_condition}, as was shown by Cascante, Ortega, and Verbitsky \cite{cascante2006}. Whereas in this special case the local averages are independent of the local integrability parameter $\gamma$ so that
$$
\Lambda_{\gamma,Q}\eqsim \Lambda_{-\infty,Q}:= \inf_{x\in Q} \big( \sum_{R\subseteq Q} \lambda_R 1_R(x) \big)\quad\text{for every $\gamma\in\br\setminus\{0\}$},$$ in the general case it is not so, and the integrability parameter turns out to be decisive, as Wolff potential conditions fail to be sufficient or necessary depending on it:

\begin{theorem}[Sufficiency and necessity of generalized Wolff potential conditions depend on $\gamma$]\label{theorem_sufficiency_and_necessity_wolff}Let $0<q< 1< p<\infty$. Let $\sigma$ and $\omega$ be locally finite Borel measures. Then the following assertions hold:
\begin{enumerate}[label=(\roman*)]
\item (Sufficiency) For every integrability parameter $\gamma \in[1,\infty)$, the generalized Wolff potential condition \eqref{eq_generalized_wolff_potential_condition} is sufficient for the two-weight norm inequality \eqref{eq_norminequality}. For each parameter $\gamma\in(0,q)$, it is not sufficient in general. 
\item (Necessity) For every integrability parameter $\gamma \in(0,q)$, the generalized Wolff potential condition \eqref{eq_generalized_wolff_potential_condition} is necessary for \eqref{eq_norminequality}. For each parameter $\gamma \in[q,\infty)$, it is not necessary in general.
\end{enumerate}
\end{theorem}

\begin{remark}[Riesz potentials] Our sufficient, necessary, or equivalent conditions  can also be applied to continuous operators that are comparable with positive model dyadic operators. 

Such comparison is well-known  for Riesz potentials.  For $\alpha\in(0,d)$, the {\it Riesz potential} of order $\alpha$, or {\it fractional integral}, $I_\alpha(\cdotroomy \sigma)$ is defined by
$$
I_\alpha(f\sigma)(x):= \int_{\br^d} \frac{f(y)}{\abs{x-y}^{d-\alpha}} \dsigma(y), 
$$
and its model dyadic operator $I^{\cd}_\alpha( \cdotroomy \sigma)$ by
$$
I^{\cd}_\alpha(f\sigma)(x):= \sum_{Q\in \cd} \frac{1_Q}{\abs{Q}^{1-\frac{\alpha}{d}}} \int_{\br^d} f(y) \dsigma(y).
$$

From the basic principle of dyadic analysis that generic cubes can be approximated by dyadic cubes of shifted dyadic systems 
$$
\cd^t:=\{2^k([0,1)^d +j + t): k\in \bz, j \in \bz^d\}, \quad \text{ $t\in \br^d$},
$$
it follows that the Riesz potential can be controlled by its model dyadic operators $I^{\cd^t}_\alpha(\cdotroomy\sigma)$ with 
 $\cd^t$ in place of $\cd$.  Consequently, the normwise comparison
$$
\norm{I_\alpha(\cdotroomy\sigma)}_{L^p(\sigma)\to L^q(\omega)}\eqsim \sup_{t\in [0,1]^d}\norm{I^{\cd^t}_\alpha(\cdotroomy\sigma)}_{L^p(\sigma)\to L^q(\omega)}
$$
holds \cite{sawyer1992,cascante2006}. The model dyadic operator $I^{\cd^t}_\alpha( \cdotroomy \sigma)$ is precisely the positive dyadic operator $T_{\lambda}( \cdotroomy \sigma)$ associated with the coefficients $\lambda:=\{\frac{\sigma(Q)}{\abs{Q}^{1-\frac{\alpha}{d}} }\}_{Q\in\cd^t}$. Thereby, our sufficient, necessary, or equivalent   conditions apply to the Riesz potential by imposing them for its model dyadic operators uniformly over the dyadic systems $\cd^t$.

Notice that in the special case that the measure $\sigma$ is Lebesgue measure, the model dyadic operator's coefficients $\lambda:=\{ \abs{Q}^{\frac{\alpha}{d}}\}_{Q\in\cd}$ satisfy the DLBO condition. In this case, the weighted norm inequality for Riesz potentials   was characterized previously by Cascante, Ortega and Verbitsky \cite{cascante2006} in terms of Wolff 
potentials, and by Maz'ya and Netrusov \cite[Sec. 11.6.1]{mazya2011} in terms of capacities. 
\end{remark}

This paper is organized as follows. In Section \ref{sec_preliminaries}, basic properties of  discrete Littlewood--Paley spaces are summarized, and Maurey's factorization theorem is discussed. In Section \ref{sec_factorization}, we prove the factorization characterization (Theorem \ref{theorem_factorization}), and obtain equivalent conditions (in particular, Theorem \ref{theorem_maintheorem}) in terms of auxiliary coefficients by using factorizations of the Littlewood--Paley spaces. In Section \ref{sec_generalized_wolff},  we prove that the generalized Wolff potential condition is necessary for small and  sufficient for large integrability parameters (one half of Theorem \ref{theorem_sufficiency_and_necessity_wolff}). By constructing concrete counterexamples, we prove that the condition is nevertheless not necessary for large parameters and not sufficient for small parameters (the other half of Theorem \ref{theorem_sufficiency_and_necessity_wolff}). In the Appendix, we summarize various integral conditions used in the paper.

We conclude the introduction by stating two open problems. In the case $q\in(0,1)$, no explicit integral conditions that characterize the two-weight norm inequality in its full generality are known, so we pose the problem: 
\begin{problem}Let $0<q< 1 \le p<\infty$. Let $\sigma$ and $\omega$ be locally finite Borel measures. Can the two-weight norm inequality \eqref{eq_norminequality} be characterized by some explicit integral conditions?
\end{problem}
In the scale of generalized Wolff potential conditions, it is unknown how large the gap between necessity and sufficiency is, so we pose the problem: 

\begin{problem}Let $0<q< 1 < p<\infty$. Let $\sigma$ and $\omega$ be locally finite Borel measures. Is the  generalized Wolff potential condition \eqref{eq_generalized_wolff_potential_condition}, 
which depends on the integrability parameter $\gamma$, sufficient for  \eqref{eq_norminequality} for some integrability parameter $\gamma \in [q,1)$? (Note that the sufficiency for $\gamma=1$ is contained in Theorem \ref{theorem_sufficiency_and_necessity_wolff}.)

\end{problem}

This research was conducted during the first author's visit to the Mathematics Department at the University of Missouri. He thanks the department for its hospitality.

\section{Preliminaries}\label{sec_preliminaries}

\subsection{Maurey's factorization}
A {\it lattice } $(E,\leq)$ is a set equipped with a partial order relation $\leq$ such that for every pair $e_1,e_2\in E$ there exists the least upper bound $e_1\vee e_2 \in E$ and the greatest lower bound $e_1\wedge e_2 \in E$.

\begin{definition}[Banach lattice]
A {\it Banach lattice} $(E,\norm{\cdotroomy}_E,\leq)$ is both a real Banach space $(E,\norm{\cdotroomy}_E)$ and a lattice $(E,\leq)$ so that both structures are compatible: 
\begin{itemize}
\item[i)]For every $e_1,e_2,e_3\in E$, we have that $e_1\leq e_2$ implies
$e_1+e_3\leq e_2+e_3$. 
\item[ii)]For every  $r\in\br$ and $e\in E$, we have that $r\geq 0$ and $e\geq 0$ implies $re\geq 0$.  
\\
The {\it positive part} $e_+$ of a vector $e\in E$ is defined by $e_+:=e\vee 0$, the {\it negative part} $e_-$ by $e_-:=-e\vee 0 $, and the {\it  absolute value} $\abs{e}$ by $\abs{e}:=e\vee -e$. 
\item[iii)]For every $e\in E$, we have that $\norm{e\,}_E=\norm{\,\abs{e}\,}_E$. For every $e_1,e_2\in E$, we have that $0\leq e_1\leq e_2$ implies $\norm{e_1\,}_E\leq \norm{e_2\,}_E$.
\end{itemize}
\end{definition}
An operator $T:E\to E'$ between Banach lattices $E$ and $E'$ is called ${\it positive}$ if for every $e\in E$ with $e\geq 0$ we have $T(e)\geq 0$, or equivalently, for every $e\in E$ we have $\abs{T(e)}\leq T(\abs{e})$.

\begin{theorem}[Maurey's factorization]\label{theorem:maurey}Let $L^q$ denote the Lebesgue space associated with  a measure space $(X,\cf,\mu)$ and an exponent $q\in(0,1)$. Let $E$ be a Banach lattice. Let $T: E\to L^q$ be a positive linear operator. Then the following assertions are equivalent:
\begin{enumerate}[label=(\roman*)]
\item\label{assertion_maurey_norm_inequality} There exists  a positive constant  $C$ such that the inequality 
$$
\big( \int \abs{T(e)}^q \dmu \big)^{\frac{1}{q}} \leq C \norm{e}_E
$$
holds for every $e\in E$.
\item \label{assertion_maurey_densityfunction}There exists a measurable function $\Phi$ with $\Phi\geq 0$, $\int \Phi \dmu \leq 1$, and $\mu$-almost everywhere $\{\Phi=0\}\subseteq \{T(e)=0\}$ for every $e\in E$, such that
$$
\int \abs{T(e)} \Phi^{-\frac{1-q}{q}}\dmu \leq C \norm{e}_E
$$
for every $e\in E$, where $C$ is a positive constant 
which does not depend on $e$.
\end{enumerate}
Furthermore, the least constants in these estimates coincide.
\end{theorem}
Note that assertion \ref{assertion_maurey_densityfunction} implies assertion \ref{assertion_maurey_norm_inequality} by H\"older's inequality, and hence the converse is the main point of the theorem. This theorem was proven by Maurey \cite[Theorem 2]{maurey1974}; an alternative proof was given by Pisier \cite[Remark on page 111]{pisier1986}.  In fact, in Maurey's book \cite[Theorem 2]{maurey1974}, the factorization theorem is (after renaming exponents and functions) phrased as follows:

\begin{theorem}[Maurey's factorization rephrased]\label{theorem:maurey_another} Let $L^q$ denote the Lebesgue space associated with  a measure space $(X,\cf,\mu)$ and an exponent $q\in(0,1)$. Let $\{f_i\}_{i\in I}$ be a family of measurable functions indexed by an index set $I$. Then the following assertions are equivalent:
\begin{enumerate}[label=(\roman*)]
\item\label{assertion_maurey_norm_inequality_family} There exists a positive constant $C$ such that the  inequality 
$$
\big( \int \big(\sum_{i\in I} \abs{\alpha_i}\abs{ f_i}\big)^q \dmu \big)^{\frac{1}{q}} \leq C \sum_{i\in I} \abs{\alpha_i}
$$
holds for every finitely supported family $\{a_i\}_{i\in I}$ of reals.
\item \label{assertion_maurey_densityfunction_family}There exists a measurable function $\Phi$ with $\Phi\geq 0$, $\int \Phi \dmu \leq 1$, and $\mu$-almost everywhere $\{\Phi=0\}\subseteq \{f_i=0\}$ for every $i\in I$, such that
$$
\int \abs{f_i} \Phi^{-\frac{1-q}{q}}\leq C
$$
for every $i\in I$.
\end{enumerate}
Furthermore, the least constants in these estimates coincide.
\end{theorem}
Theorem \ref{theorem:maurey} and Theorem \ref{theorem:maurey_another} are clearly equivalent: 
\begin{itemize}
\item Applying Theorem \ref{theorem:maurey} to the Banach lattice $E:=\ell^1(I)$ and the positive linear operator $T(\{a_i\}_{i\in I}):=\sum_{i\in I} \alpha_i \abs{f_i}$ yields Theorem \ref{theorem:maurey_another}.
\item Conversely, applying Theorem \ref{theorem:maurey_another} to the family $\{\frac{T(e)}{\norm{e}_E}\}_{e\in E\setminus\{0\}}$ of functions yields Theorem \ref{theorem:maurey} because in this case, by the assumed positivity and linearity of the operator $T$, we have
$$
\sum_{e\in E\setminus\{0\}} \abs{\alpha_e} \, \abs{\frac{T(e)}{\norm{e}_E}}=  \sum_{e\in E\setminus\{0\}} \abs{\frac{T(\abs{\alpha_e}e)}{\norm{e}_E}}\leq \sum_{e\in E\setminus\{0\}} T( \frac{\abs{\alpha_e} \, \abs{e}}{\norm{e}_E})=T( \sum_{e\in E\setminus\{0\}} \frac{\abs{\alpha_e}\, \abs{e}}{\norm{e}_E}) 
$$
and hence assertion (i) of Theorem \ref{theorem:maurey_another} is equivalent to assertion (i) of Theorem \ref{theorem:maurey}
\end{itemize}

\subsection{Discrete Littlewood--Paley spaces}\label{sec:preliminaries_littlewoodpaleynorms}
\begin{definition}[Discrete Littlewood--Paley norms]Let $\mu$ be a locally finite Borel measure, and let $p\in(0,\infty]$ and $q\in(0,\infty]$. Let $a:=\{a_Q\}$ be a family of non-negative reals. The {\it discrete Littlewood--Paley norm} $\norm{\cdotroomy}_{f^{p,q}(\mu)}$ is defined  as follows:
\begin{itemize}

\item For $p\in(0,\infty)$ and $q\in\br\setminus\{0\}$,
$$
\norm{a}_{f^{p,q}(\mu)}:= \Big(\int \big(  \sum_Q a_Q^q 1_Q\big)^{\frac{p}{q}} \dmu\Big)^{\frac{1}{p}}.
$$
\item For $p\in(0,\infty)$ and $q=\infty$,
$$
\norm{a}_{f^{p,\infty}(\mu)}:= \Big(\int \big( \sup_Q a_Q 1_Q \big)^p \dmu\Big)^{\frac{1}{p}}.
$$
\item For $p=\infty$ and $q\in\br\setminus\{0\}$,
$$
\norm{a}_{f^{\infty,q}(\mu)}:= \sup_{Q }\Big( \frac{1}{\mu(Q)} \sum_{R\subseteq Q} a_R^q \mu(R) \Big)^{\frac{1}{q}}.
$$
\item For $p=\infty$ and $q=\infty$,
$$
\norm{a}_{f^{\infty,\infty}(\mu)}:=\sup_Q a_Q.
$$
\end{itemize}
\end{definition}
 We use the notation ${f^{p,q}}$  in place of $f^{p,q}(\mu)$ if the measure $\mu$ has been specified in the context. 

Note that, except at the endpoint $p=\infty$, the discrete Littlewood--Paley norm $\norm{\cdotroomy}_{f^{p,q}}$ is just the mixed Lebesgue norm $\norm{\cdotroomy}_{L^p( \ell^q)}$, whereas at the endpoint $p=1$, it  is the Carleson norm instead of the mixed Lebesgue norm $\norm{\cdotroomy}_{L^\infty (\ell^q)}$.

Note also that the discrete Littlewood--Paley norm has the scaling property: 
\begin{equation}
\norm{\{a_Q^s\}}_{f^{p,q}}=\norm{\{a_Q\}}_{f^{sp,sq}}\quad \text{for every $s\in(0,\infty)$.}
\end{equation}

The discrete Littlewood--Paley norm can be computed via  duality \cite[Theorem 4  and Remark 5]{verbitsky1996}:

\begin{proposition}[Computing norm by duality]\label{proposition_normduality} Let $p,q\in[1,\infty]$. Let $\mu$ be a locally finite Borel measure.  Then we have
$$
\norm{a}_{f^{p,q}}\eqsim_{p,q} \sup_{\norm{b}_{f^{p',q'}}\leq 1} \sum_Q a_Q b_Q \mu(Q).
$$
\end{proposition}

The following factorization  theorem 
was proven by Cohn and Verbitsky \cite[Theorem 2.4]{cohn2000}. 
\begin{proposition}[Factorization of discrete Littlewood--Paley spaces]
\label{proposition:factorization_littlewood}Let $\mu$ be a locally finite Borel measure on $\br^d$. Let $p,p_1,p_2\in(0,\infty]$ and $q,q_1,q_2\in(0,\infty]$ be exponents that satisfy the H\"older relations:
$$
\frac{1}{p}=\frac{1}{p_1}+\frac{1}{p_2}\quad \text{ and } \quad\frac{1}{q}=\frac{1}{q_1}+\frac{1}{q_2}. 
$$
Then the following assertions hold:
\begin{enumerate}[label=(\roman*)]
\item Every $a\in f^{p_1,q_1}$ and $b\in f^{p_2,q_2}$ satisfy the estimate
$$
\norm{ab}_{f^{p,q}}\lesssim_{q,p}\norm{a}_{f^{p_1,q_1}}\norm{b}_{f^{p_2,q_2}}.
$$
\item For each $c\in f^{p,q}$ there exists $a\in f^{p_1,q_1}$ and $b\in f^{p_2,q_2}$ such that $c=ab$ and
$$
\norm{a}_{f^{p_1,q_1}} \, \norm{b}_{f^{p_2,q_2}} \lesssim_{p,q} \norm{c}_{f^{p,q}}.
$$
\end{enumerate}
\end{proposition}

The factorization $f^{p,q}=f^{p_1,q_1}f^{p_2,q_2}$ is actually deduced by combining the factorization $f^{p,q}=f^{p,\infty} f^{\infty,q}$ (which is \cite[Theorem 2.4]{cohn2000}) with the trivial factorizations $f^{\infty,q}=f^{\infty,q_1}f^{\infty,q_2}$ and $f^{p,\infty}=f^{p_1,\infty}f^{p_2,\infty}$.

\subsection{Equivalent discrete expressions}

\begin{lemma}[Summation by parts]\label{lemma_summationbyparts}For every $p\in(0,\infty)$, we have
$$
\big( \sum_Q a_Q 1_Q\big)^{p}\eqsim_p \sum_Q a_Q \big(\sum_{R\subseteq Q} a_R 1_R \big)^{p-1}\eqsim_p \sum_Q a_Q 1_Q \big(\sum_{R\supseteq Q} a_R 1_R \big)^{p-1} .
$$
\end{lemma}
\begin{proof}At each point, the summation is linearly ordered by the nestedness of dyadic cubes. Using the mean value theorem (applied to the function $t\mapsto t^p$) and summation by parts yields the estimates.
\end{proof}
\begin{remark}Applying the lemma to the summation inside the integration, we deduce, 
for every $p\in(0,\infty)$, 
\begin{equation}\label{equation:summation_by_parts_inside_integral}
\int \big( \sum_Q a_Q 1_Q\big)^{p} \dmu \eqsim_p \sum_Q a_Q \mu(Q) \big(\sum_{R\supseteq Q} a_R \big)^{p-1}\eqsim_p \sum_Q a_Q \int_Q \big(\sum_{R\subseteq Q} a_R 1_R \big)^{p-1} \dmu.
\end{equation}
\end{remark}

The following lemma was proven in \cite[Proposition 2.2]{cascante2004}:
\begin{lemma}[Equivalent discrete expressions]\label{lemma:equivalent_expressions}Let $p\in(1,\infty)$. Then the following expressions are comparable:
\begin{equation}
\begin{split}
&\int \big( \sum_Q a_Q 1_Q \big)^p \dmu\\
&\eqsim_p \sum_Q a_Q \mu(Q) \big(\frac{1}{\mu(Q)} \sum_{R\subseteq Q} a_R \mu(R) \big)^{p-1}\\
&\eqsim_p \int \big( \sup_{Q} \frac{1_Q}{\mu(Q)} \sum_{R\subseteq Q} a_R \mu(R) \big)^p \dmu.
\end{split}
\end{equation}
\end{lemma}

\begin{remark}From the lemma together with H\"older's inequality it follows that we also have
\begin{equation}
\label{eq_equivalent_norm_sup}
\int \big( \sum a_Q 1_Q \big)^p \dmu \eqsim_p \sum_Q a_Q \mu(Q) \big( \sup_{R\supseteq Q} \frac{1}{\mu(R)} \sum_{S\subseteq R} a_S \mu(S) \big)^{p-1}.
\end{equation}
\end{remark}

\subsection{Straightforward sufficient, necessary, and equivalency conditions}\label{sec_straightforward_conditions} We collect
straightforward necessary, sufficient, or equivalent  conditions in the following lemma. Assertions \ref{item_1} and \ref{item_3} were observed by Cascante and Ortega \cite[Proof of Theorem 1.1]{cascante2017}.
\begin{lemma}[Sufficient, necessary, or equivalent conditions]\label{lemma_immediateconditions}The following assertions hold:
\begin{enumerate}[label=(\roman*)]
\item \label{item_1}For $p\in(1,\infty)$ and $q\in(0,\infty)$, the norm estimate \eqref{eq_norminequality} is equivalent to the estimate
\begin{equation}
\label{eq_norm_inequality_sup}
\norm{\sum_Q \lambda_Q a_Q 1_Q}_{L^q(\omega)}\lesssim \norm{\sup_Q a_Q 1_Q}_{L^{p}(\sigma)}\quad \text{for every $\{a_Q\}$}.
\end{equation}

\item \label{item_2} \label{item_sup_equiv_sum } For every $p,q\in(0,\infty)$, estimate \eqref{eq_norm_inequality_sup} is equivalent to the estimate
\begin{equation}\label{eq_norm_inequality_sum}
\norm{\sum_Q \rho_Q b_Q 1_Q}_{L^q(\omega)}\lesssim \norm{\sum_Q b_Q 1_Q}_{L^{p}(\sigma)}\quad \text{for every $\{b_Q\}$},
\end{equation}
where $\rho_Q$ is defined as the localized sum $\rho_Q:=\sum_{R\subseteq Q} \lambda_R 1_R$.
\item \label{item_3}In terms of the Littlewood--Paley norms, estimate \eqref{eq_norm_inequality_sup} can be viewed as the $f^{\frac{p}{q},\infty}(\sigma)\to f^{1,\frac{1}{q}}(\omega)$ norm estimate for the multiplier operator $\{a_Q\}\mapsto \{\lambda_Q^q a_Q\}$, which is the estimate
\begin{equation}
\norm{\{\lambda_Q^q a_Q\}}_{f^{1,\frac{1}{q}}(\omega)}\lesssim \norm{a}_{f^{\frac{p}{q},\infty}(\sigma)}.
\end{equation}
By duality, it is equivalent to the $f^{\infty,\frac{1}{1-q}}(\omega)\to f^{\frac{p}{p-q},1}(\sigma) $ norm estimate  for the adjoint multiplier operator $\{b_Q\}\mapsto \{\lambda_Q^q \frac{\omega(Q)}{\sigma(Q)} b_Q\}$, i.e., the estimate
\begin{equation}\label{eq_norminequality_dual}
\norm{\{\lambda_Q^q \frac{\omega(Q)}{\sigma(Q)} b_Q\}}_{f^{\frac{p}{p-q},1}(\sigma)}\lesssim \norm{b}_{f^{\infty,\frac{1}{1-q}}(\omega)}.
\end{equation}
\item \label{item_4} The condition 
\begin{equation}
\label{eq_sufficient_dualeasy}
\int \big( \sum_Q \lambda_Q^q \frac{\omega(Q)}{\sigma(Q)} 1_Q \big)^{\frac{p}{p-q}} \dsigma< \infty
\end{equation}
is sufficient for \eqref{eq_norminequality_dual}.
\item \label{item_5}
For every $\gamma \in(0,q)$, the estimate
\begin{equation}
\label{eq_necessary_stein}
\norm{\sum_Q \Lambda_{\gamma,Q} a_Q 1_Q}_{L^q(\omega)}\lesssim_{\gamma,q} \norm{\sum_Q a_Q 1_Q}_{L^{p}(\sigma)}
\end{equation}
is necessary for \eqref{eq_norm_inequality_sum}.
\end{enumerate}
\end{lemma}
\begin{proof}
{\it\ref{item_1}.} One direction follows from substituting $f:=\sup_Q a_Q 1_Q$,  and the other from substituting $a_Q:=\angles{f}^\sigma_Q$ and using the dyadic Hardy--Littlewood maximal inequality.  

{\it \ref{item_2}.} First, we observe that for each $b$ there exists $a$ (which depends on $b$) that satisfies the relations
\begin{equation}\label{eq:domination1}
\sum_Q \rho_Q b_Q 1_Q = \sum_Q \lambda_Q a_Q 1_Q \quad\text{ and } \quad \sup_Q a_Q 1_Q = \sum_Q b_Q 1_Q;
\end{equation}
indeed, the choice $a_Q:=\sum_{R\supseteq Q} b_Q$ works. Next, we observe that for each $a$ there exists $b$ (which depends on $a$) that satisfies the relations

\begin{equation}\label{eq:domination2}
\sum_Q \lambda_Q a_Q 1_Q \leq \sum_Q \rho_Q b_Q 1_Q \quad\text{ and } \quad \sum_Q b_Q 1_Q=\sup_Q a_Q 1_Q;
\end{equation}
indeed, the choice $$b_Q:= \sum_{R\supseteq Q} \Big(\big(\sup_{S\supseteq R} a_S\big) -  \big(\sup_{S\supseteq \hat{R}}a_S\big)\Big),$$ where $\hat{Q}$ denotes the dyadic parent of the dyadic cube $Q$, works by telescoping summation. Estimate \eqref{eq_norm_inequality_sup} implies estimate \eqref{eq_norm_inequality_sum} through the relations \eqref{eq:domination1}, and conversely, estimate \eqref{eq_norm_inequality_sum} implies estimate \eqref{eq_norm_inequality_sup} through the relations \eqref{eq:domination2}.

{\it \ref{item_3}.} This assertion follows by writing estimate \eqref{eq_norm_inequality_sup} in terms of the discrete Littlewood--Paley spaces and using duality (Proposition \ref{proposition_normduality}).

{\it \ref{item_4}.} The sufficiency of condition \eqref{eq_sufficient_dualeasy} follows from the dual estimate \eqref{eq_norminequality_dual} together with the  trivial estimate $\sup_Q b_Q\leq \norm{b}_{f^{\infty,\frac{1}{1-q}}(\omega)}$.

{\it \ref{item_5}.}
The necessity of condition \eqref{eq_necessary_stein} follows from estimating the left-hand side of inequality \eqref{eq_norm_inequality_sum} from below by using the scaling of the $L^p$ norms and Stein's inequality:
\begin{equation*}
\begin{split}
&\norm{\sum_Q \big( \sum_{Q\subseteq R} \lambda_R 1_R \big) a_Q 1_Q}_{L^q(\omega)}=\norm{\Big(\sum_Q \Big( \big(( \sum_{Q\subseteq R} \lambda_R 1_R ) a_Q\big)^{\gamma}\Big)^{ \frac{1}{\gamma}} 1_Q \Big)^{\gamma}}_{L^{\frac{q}{\gamma}(\omega)}}^{\frac{1}{\gamma}} \\
& \gtrsim_{\gamma,q}  \norm{\sum_Q \Big(\frac{1}{\omega(Q)}\int \big(\sum_{R\subseteq Q} \lambda_Q 1_Q \big)^\gamma\domega \Big)^{\frac{1}{\gamma}} a_Q 1_Q}_{L^q(\omega)}=: \norm{\sum_Q \Lambda_{\gamma,Q} a_Q 1_Q}_{L^q(\omega)}.
\end{split}
\end{equation*}

\end{proof}

\section{Characterization by factorization}\label{sec_factorization}
\subsection{Factorization condition}In this section, we prove Theorem \ref{theorem_factorization}.

First, we apply Maurey's factorization (Theorem \ref{theorem:maurey}). By applying it to the positive linear operator $T(\cdotroomy \sigma)$ from the Banach lattice $L^p(\sigma)$ into the Lebesgue space $L^q(\omega)$, we see that the two-weight  norm inequality \eqref{eq_norminequality} is equivalent to the existence of a Borel measurable function $\Phi\geq 0$ such that
\begin{subequations}
\begin{align}
\nonumber&\int \Phi \domega \leq 1 \\ 
\label{eq_fact_secondcondition}&\int \big(\sum_{Q\in \cd} \lambda_Q \angles{f}^\sigma_Q 1_Q \big) \Phi^{-\frac{1-q}{q}} \domega \leq C \norm{f}_{L^p(\sigma)}. 
\end{align}
\end{subequations}
Furthermore, we have $\{\Phi=0\}\subseteq \{T(f\sigma)=0\}$ for every $f\in L^p(\sigma)$, which 
means 
\begin{equation}\label{eq_fact_zerodivision}
\text{if $\lambda_Q>0$ and $\omega(Q)>0$, then $\Phi>0$ $d\omega$-a.e. on $Q$.} 
\end{equation}
This condition guarantees that no division by zero occurs, as we may assume that all the cubes $Q$ with $\lambda_Q=0$ or $\omega(Q)=0$ (or $\sigma(Q)=0$) are omitted from the summation because such cubes do not contribute to inequality \eqref{eq_norminequality}. From now on we restrict the indexation to be over the collection $\cq$ of the remaining cubes
$$
\cq:=\{Q\in \cd : \, \lambda_Q>0, \sigma(Q)>0, \text{ and }  \omega(Q)>0\}.
$$

By interchanging  the order of integration and summation in \eqref{eq_fact_secondcondition} and  using the $L^p(\sigma)-L^{p'}(\sigma)$ duality,  we see that \eqref{eq_fact_secondcondition} is  
equivalent to 
$$
\big( \int \big(\sum_{Q\in \cq} \lambda_Q \angles{\Phi^{-\frac{1-q}{q}}}^\omega_Q \frac{\omega(Q)}{\sigma(Q)}1_Q \big)^{p'}  \dsigma \big)^{\frac{1}{p'}}\leq C.
$$
 By  \eqref{eq_fact_zerodivision} together with the remark following it, the average $\angles{\Phi^{-\frac{1-q}{q}}}^\omega_Q$ is positive for every $Q\in \cq$.

Next, we discretize. We prove that the following assertions are equivalent:
\begin{enumerate}[label=(\roman*)]
\item \label{assertion_fact_continuous} There exists a function $\Phi$, with $\Phi>0$ $d\omega$-a.e.  on every cube $Q\in \cq$, that satisfies the pair of conditions
\begin{subequations}
\begin{align}
\label{eq_fact_continuous1}&\int \Phi \domega \lesssim_q 1 \\ 
\label{eq_fact_continuous2}&\big( \int \big(\sum_{Q\in \cq} \lambda_Q \angles{\Phi^{-\frac{1-q}{q}}}^\omega_Q \frac{\omega(Q)}{\sigma(Q)}1_Q \big)^{p'}  \dsigma \big)^{\frac{1}{p'}}\lesssim_q C.
\end{align}
\end{subequations}

\item \label{assertion_fact_discrete} There exists a family $\{a_Q\}_{Q\in \cq}$ of positive reals that satisfies the pair of conditions

\begin{subequations}\label{eq_fact_discrete}
\begin{align}
\label{eq_fact_discrete1}&\int \big(\sup_{Q\in \cq} {a_Q 1_Q}\big)^{\frac{q}{1-q}} \domega \lesssim_1 1 \\ 
\label{eq_fact_discrete2}&\big( \int \big(\sum_{Q\in \cq} \lambda_Q a_Q^{-1} \frac{\omega(Q)}{\sigma(Q)}1_Q \big)^{p'}  \dsigma \big)^{\frac{1}{p'}}\lesssim_q C.
\end{align}
\end{subequations}
\end{enumerate}

First, we prove that the continuous conditions imply the discrete ones. We set $$a_Q^{-1}:= \angles{\Phi^{-\frac{1-q}{q}}}^\omega_Q$$ for every cube $Q\in\cq$. Thus,  condition \eqref{eq_fact_discrete2} becomes condition \eqref{eq_fact_continuous2} . By Jensen's inequality together with the convexity of the function $t\mapsto t^{-q}$, and the Hardy--Littlewood maximal inequality, condition  \eqref{eq_fact_continuous1} implies condition \eqref{eq_fact_discrete1} through
\begin{equation*}
\begin{split}
&\int \big(\sup_{Q\in \cq} {a_Q 1_Q}\big)^{\frac{q}{1-q}} \domega= \int \sup_{Q\in \cq} \big((\angles{\Phi^{-\frac{1-q}{q}}}^\omega_Q)^{-q}\big)^{\frac{1}{1-q}} \domega\\
&\leq \int \big(\sup_{Q\in \cq} \angles{\Phi^{1-q}}^\omega_Q 1_Q )^{\frac{1}{1-q}} \domega \lesssim_q \int \Phi \domega.
\end{split}
\end{equation*}

Next, we prove that the discrete conditions imply the continuous ones. We set $$\Phi:=\big(\sup_{Q\in \cq} {a_Q 1_Q}\big)^{\frac{q}{1-q}}.$$ Thus, condition \eqref{eq_fact_continuous1} becomes  condition \eqref{eq_fact_discrete1} . By estimating the supremum from below by omitting all but one cube from the indexation, we see that condition \eqref{eq_fact_discrete2} implies  condition \eqref{eq_fact_continuous2}. The proof is complete.

\subsection{Related equivalent conditions}For a family $a:=\{a_Q\}$ of positive reals, we write $a^{-1}:=\{a_Q^{-1}\}$. All the indexations throughout this section are restricted to the subcollection
$$
\cq:=\{Q\in \cd : \lambda_Q>0, \sigma(Q)>0, \text{ and }  \omega(Q)>0\}
$$
of dyadic cubes, and hence no division by zero occurs.  We abbreviate the indexation  \quotes{$Q\in \cq$} as \quotes{$Q$}.

For every family $a:=\{a_Q\}$ of positive reals, we define the quantities $A_1(a^{-1})$ and $A_2(a)$, and  conditions \eqref{conditions_a} by
\begin{subequations}\label{conditions_a}
\begin{align}
\label{cond_a1} A_1(a^{-1})&:=\Big(\int \big( \sum_Q \lambda_Q \frac{\omega(Q)}{\sigma(Q)} a_Q^{-1} 1_Q \big)^{p'} \dsigma\Big)^{\frac{1}{p'}}  < \infty \\
 \label{cond_a2}A_2(a)&:= \Big( \int \big( \sup_Q a_Q 1_Q \big)^{\frac{q}{1-q}} \domega\Big)^{\frac{1-q}{q}}< \infty.
\end{align}
\end{subequations}
For every family  $d:=\{d_Q\}$ of positive reals, we define the quantities $D_1(d^{-1})$ and $D_2(d)$, and  conditions \eqref{conditions_d} by
\begin{subequations}\label{conditions_d}
\begin{align}
 \label{cond_d1}  D_1(d^{-1})&:= \sup_Q \frac{1}{\sigma(Q)} \sum_{R\subseteq Q} \lambda_Q d_Q^{-1} \omega(Q)  <\infty \\
\label{cond_d2} D_2(d) &:=  \Big(\int \big( \sum_Q \lambda_Q d_Q^{p'-1} 1_Q \big)^{\frac{(p-1)q}{p-q}}  \domega\Big)^{\frac{p-q}{q}}<\infty.
\end{align}
\end{subequations}

In this section, we prove the following proposition:
\begin{proposition}\label{proposition_equivalenceofauxiliaryfamilies}The following assertions hold:

\begin{enumerate}[label=(\roman*)]
\item\label{item_existence_a} For each family $a:=\{a_Q\}$ of positive reals that satisfies  conditions  \eqref{conditions_a}, there exists a family $d:=\{d_Q\}$ of positive reals that satisfies conditions \eqref{conditions_d}; in fact, such a family is given by
$$
d_Q:= a_Q \sup_{R\supseteq Q} \big( \frac{1}{\sigma(R)} \sum_{S\subseteq R} \lambda_S a_S^{-1} \omega(S) \big),
$$
and satisfies the estimates:
\begin{subequations}
\begin{align}
D_1(d^{-1}) &\lesssim 1\\
D_2(d)&\lesssim  A_1(a^{-1})^p A_2(a)^p.
\end{align}
\end{subequations}

\item \label{item_existence_d}For each family $d:=\{d_Q\}$ of positive reals that satisfies  conditions \eqref{conditions_d}, there exists a family $a:=\{a_Q\}$  of positive reals that satisfies conditions \eqref{conditions_a}; in fact, such a family is given by 
$$
a_Q:=\big( \sum_{R\supseteq Q} \lambda_R d_R^{p'-1}\big)^{\frac{(1-q)(p-1)}{p-q}}.
$$ and satisfies the estimates:
\begin{subequations}
\begin{align}
A_1(a^{-1})&\lesssim D_1(d^{-1})^{\frac{1}{p}} D_2(d)^{\frac{(p-1)q}{p(p-q)}}\\
 A_2(a)&\lesssim  D_2(d)^{\frac{1-q}{p-q}}.
\end{align}
\end{subequations}

\end{enumerate}

\end{proposition}
We note that, by the proposition, for each family $a$ there exists a family $d$ such that the estimate $$(D_1(d^{-1}))^{\frac{1}{p}}(D_2(d))^{\frac{1}{p}}\lesssim A_1(a^{-1}) A_2(a)$$ holds, and, conversely, for each family $d$ there exists a family $a$ such that the reverse estimate $$A_1(a^{-1}) A_2(a)\lesssim (D_1(d^{-1}))^{\frac{1}{p}}(D_2(d))^{\frac{1}{p}}$$ holds. Combining these estimates with the characterization by factorization (Theorem \ref{theorem_factorization}) yields the characterization by auxiliary coefficients (Theorem \ref{theorem_maintheorem}).

To prepare for the proof of the proposition, we split each of the conditions \eqref{cond_d2} and  \eqref{cond_a1} into equivalent subconditions by writing out factorizations in the Littlewood--Paley spaces. 
\begin{lemma}[Factorization of condition \eqref{cond_d2}]\label{lemma_d2_sub} Let $d:=\{d_Q\}$ be a family of positive reals. Then the following assertions hold:

\begin{enumerate}[label=(\roman*)]
\item Every family $e:=\{e_Q\}$ of positive reals satisfies  the estimate
\begin{equation}\label{eq_upperde}
\begin{split}
&\Big(\int \big( \sum_Q \lambda_Q d_Q^{p'-1} 1_Q \big)^{\frac{(p-1)q}{p-q}} \domega\Big)^{\frac{p-q}{q}} \\
&\lesssim \Big(\sum_Q \lambda_Q e_Q^{-p'} \omega(Q) d_Q^{p'-1}\Big)^{p-1} \Big(\int \big( \sup e_Q 1_Q \big)^{\frac{q}{1-q}} \Big)^{\frac{p(1-q)}{q}}.
\end{split}
\end{equation}
\item Some family $e:=\{e_Q\}$  of positive reals (which depends on the family $d:=\{d_Q\}$) satisfies the reverse of  estimate \eqref{eq_upperde}.

\end{enumerate}
\end{lemma}
\begin{proof}
The lemma follows from writing out the factorization $$f^{1,\frac{p-q}{(p-1)q}}(\omega)= f^{\frac{p-q}{(p-1)q},\frac{p-q}{(p-1)q}}(\omega)\cdot f^{\frac{p-q}{p(1-q)}, \infty}(\omega)$$ of the Littlewood--Paley spaces (stated in Proposition \ref{proposition:factorization_littlewood}).
\end{proof}

\begin{lemma}[Factorization of condition \eqref{cond_a1}]\label{lemma_asub1}

Let $a^{-1}:=\{a^{-1}_Q\}$ be a family of positive reals. Then the following assertions hold:

\begin{enumerate}[label=(\roman*)]
\item Every family $b:=\{b_Q\}$  of positive reals satisfies the estimate
\begin{equation}\label{eq_upperab}
\begin{split}
&\Big(\int \big( \sum_Q \lambda_Q \frac{\omega(Q)}{\sigma(Q)} a_Q^{-1} 1_Q \big)^{p'} \dsigma\Big)^{\frac{1}{p'}}\\
&\lesssim 
\Big(\sup_Q \frac{1}{\sigma(Q)}\sum_{R\subseteq Q} \lambda_R b_R^{-1} \omega(R)\Big)^{\frac{1}{p}}
\Big(\sum_Q \lambda_Q a_Q^{-p'} b_Q^{p'-1} \omega(Q)\Big)^{\frac{1}{p'}} .
\end{split}
\end{equation}
\item Some family $b:=\{b_Q\}$ of positive reals (which depends on the family $a:=\{a_Q\}$) satisfies the reverse of estimate \eqref{eq_upperab}.

\end{enumerate}

\end{lemma}
\begin{proof}
The lemma follows from writing out the factorization $$f^{p',1}(\sigma)= f^{p',p'}(\sigma)\cdot f^{\infty,p}(\sigma)$$ of the Littlewood--Paley spaces (stated in Proposition \ref{proposition:factorization_littlewood}).
\end{proof}

We are now prepared for the proof of the proposition:

\begin{proof}[Proof of Proposition \ref{proposition_equivalenceofauxiliaryfamilies}]
First, we prove  assertion \ref{item_existence_d}. Assume that $d$ is a family that satisfies the conditions 
\begin{subequations}
\begin{align}
&
\sup_Q \frac{1}{\sigma(Q)} \sum_{R\subseteq Q} \lambda_Q d_Q^{-1} \omega(Q) < \infty \label{condition_da} \\
&
\int \big( \sum_Q \lambda_Q d_Q^{p'-1} 1_Q \big)^{\frac{(p-1)q}{p-q}}  \domega < \infty\label{condition_db} 
\end{align}
\end{subequations}
Since, by Lemma \ref{lemma:equivalent_expressions}, we have $$ \sum_Q \lambda_Q d_Q^{p'-1} \omega(Q) \big( \sum_{R\supseteq Q}\lambda_R d_R^{p'-1} \big)^{-\frac{p(1-q)}{p-q}}\eqsim \int \big( \sum_Q \lambda_Q d_Q^{p'-1} 1_Q \big)^{\frac{(p-1)q}{p-q}},$$ 
 condition \eqref{condition_db} is equivalent to the condition
\begin{equation}
\label{cond_d2variant}
\sum_Q \lambda_Q d_Q^{p'-1} \omega(Q) \big( \sum_{R\supseteq Q}\lambda_R d_R^{p'-1} \big)^{-\frac{p(1-q)}{p-q}} < \infty
\end{equation}
By factorization (Lemma \ref{lemma_asub1}), it is sufficient (and necessary) to construct families $a$ and $b$ (which depend on the family $d$) that satisfy the conditions
\begin{subequations}
\begin{align}
&\ \sup_Q \frac{1}{\sigma(Q)}\sum_{R\subseteq Q} \lambda_R b_R^{-1} \omega(R) \infty\label{condition_ba}\\
& \sum_Q \lambda_Q a_Q^{-p'} b_Q^{p'-1} \omega(Q) < \infty \label{condition_bb}\\
& \int \big( \sup_Q a_Q 1_Q \big)^{\frac{q}{1-q}} \domega < \infty \label{condition_bc}
\end{align}
\end{subequations}
Comparing condition \eqref{cond_d2variant} with condition \eqref{condition_bb}, we set
$$
a_Q^{-p'}:= \big( \sum_{R\supseteq Q}\lambda_R d_R^{p'-1} \big)^{-\frac{p(1-q)}{p-q}} \quad \text{ and } \quad b_Q:=d_Q.
$$
With these choices, condition \eqref{condition_bb} coincides with condition \eqref{cond_d2variant},  condition \eqref{condition_bc} with  condition \eqref{condition_db}, and  condition \eqref{condition_ba} with condition \eqref{condition_da}. 

Next, we prove assertion \ref{item_existence_a}.  Assume that $a$ is a family that satisfies the conditions 
\begin{subequations}
\begin{align}
&
 \int \big( \sum_Q \lambda_Q \frac{\omega(Q)}{\sigma(Q)} a_Q^{-1} 1_Q \big)^{p'} \dsigma < \infty \label{condition__aa}\\
& \int \big( \sup_Q a_Q 1_Q \big)^{\frac{q}{1-q}} \domega  < \infty\label{condition__ab}
\end{align}
\end{subequations}
Since, by the comparison \eqref{eq_equivalent_norm_sup}, we have $$\int \big( \sum_Q \lambda_Q \frac{\omega(Q)}{\sigma(Q)} a_Q^{-1} 1_Q \big)^{p'} \dsigma\eqsim \sum_Q \lambda_Q a_Q^{-1} \omega(Q) \big( \sup_{R\supseteq Q} \frac{1}{\sigma(R)} \sum_{S\subseteq R} \lambda_S a_S^{-1} \omega(S) \big)^{p'-1},$$
condition \eqref{condition__aa} is equivalent to the condition
\begin{equation}
\label{condition__aa_variant}  \sum_Q \lambda_Q a_Q^{-1} \omega(Q) \big( \sup_{R\supseteq Q} \frac{1}{\sigma(R)} \sum_{S\subseteq R} \lambda_S a_S^{-1} \omega(S) \big)^{p'-1} < \infty .
\end{equation}
By factorization (Lemma \ref{lemma_d2_sub}), it is sufficient (and necessary) to construct families $d$ and $e$ (which depend on the family $a$) that satisfy
the conditions

\begin{subequations}
\begin{align}
& \sup_Q \frac{1}{\sigma(Q)} \sum_{R\subseteq Q} \lambda_Q d_Q^{-1} \omega(Q)   < \infty\label{condition__ea}\\
 & \sum_Q \lambda_Q e_Q^{-1} \omega(Q) (e_Q^{-1} d_Q )^{p'-1} < \infty \label{condition__eb}\\
 & \int \big( \sup e_Q 1_Q \big)^{\frac{q}{1-q}} \domega < \infty\label{condition__ec}.
\end{align}
\end{subequations}

Comparing condition \eqref{condition__aa_variant}  with condition \eqref{condition__eb}, we set
$$
e_Q:= a_Q \quad \text{ and } \quad e_Q^{-1}d_Q:=\big( \sup_{R\supseteq Q} \frac{1}{\sigma(R)} \sum_{S\subseteq R} \lambda_S a_S^{-1} \omega(S) \big).  
$$
With these choices, condition \eqref{condition__eb} becomes condition \eqref{condition__aa_variant}, condition \eqref{condition__ec} becomes condition \eqref{condition__ab}. Condition \eqref{condition__ea} also holds because, by omitting all but one cube from the supremum, we have:
\begin{equation*}
\begin{split}
&\sum_{Q\subseteq P} \lambda_Q d_Q^{-1} \omega(Q)= \sum_{Q\subseteq P} \frac{\lambda_Q a_Q^{-1} \omega(Q)}{\sup_{R\supseteq Q} \frac{1}{\sigma(R)} \sum_{S\subseteq R} \lambda_S a_S^{-1} \omega(S)}\\
&\leq \sum_{Q\subseteq P} \frac{\lambda_Q a_Q^{-1} \omega(Q)}{ \frac{1}{\sigma(P)} \sum_{S\subseteq P} \lambda_S a_S^{-1} \omega(S)} \leq \sigma(P).
\end{split}
\end{equation*}

The claimed comparisons of the appropriate powers of the quantities $A_1,A_2$ and $D_1,D_2$ can be seen from  Lemma \ref{lemma_d2_sub} and Lemma \ref{lemma_asub1}, or alternatively, the appropriate powers can be determined by matching the homogeneity in the comparisons under the scaling with respect to the family $a$ or $d$, the family $\lambda$, and the measures $\sigma$ and $\omega$. The proof is complete.
\end{proof}

We conclude this section by recording the following factorization of condition \eqref{cond_a1}.

\begin{lemma}[Another factorization of condition \eqref{cond_a1}]\label{lemma:subconditions_trivial_factorization}

Let $a^{-1}:=\{a^{-1}_Q\}$ be a family of positive reals. Then the following assertions hold:

\begin{enumerate}[label=(\roman*)]
\item Every family $c:=\{c_Q\}$ of positive reals satisfies the estimate
\begin{equation}\label{eq_upperac}
\begin{split}
&\Big(\int \big( \sum_Q \lambda_Q \frac{\omega(Q)}{\sigma(Q)} a_Q^{-1} 1_Q \big)^{p'} \dsigma\Big)^{\frac{1}{p'}}\\
&\lesssim 
\Big(\sup_Q \frac{c_Q^{-1}}{a_Q^{-1}\sigma(Q)} \sum_{R\subseteq Q} \lambda_R a_R^{-1}\omega(R)\Big)^{\frac{1}{p'}}
\Big(\sum_Q \lambda_Q a_Q^{-p'} c_Q^{p'-1} \omega(Q) \Big)^{\frac{p-1}{p'}}.
\end{split}
\end{equation}
\item Some family $c:=\{c_Q\}$ of positive reals (which depends on the family $a:=\{a_Q\}$) satisfies the reverse of estimate \eqref{eq_upperac}.

\end{enumerate}

\end{lemma}

\begin{proof}We make the following trivial observation: for every families $\{a_i\}$ and $\{b_i\}$ of positive reals, we have that $\sum_i a_i b_i \leq C $ if and only there exists a family $\{c_i\}$ of positive reals such that $\sum_i a_i c_i \leq C$ and  $ \sup_{i} b_i c_i^{-1} \leq 1 $. Using Lemma \ref{lemma:equivalent_expressions}, we write
$$
\int \big( \sum_Q \lambda_Q \frac{\omega(Q)}{\sigma(Q)} 1_Q \big)^{p'} \dsigma \eqsim \sum_Q \lambda_Q \omega(Q)  \Big( \frac{1}{\sigma(Q)} \sum_{R\subseteq Q} \lambda_R a_R^{-1} \omega(R) \big)^{p'-1}
$$
Applying this trivial observation to the summation on the right hand-side yields the lemma.
\end{proof}
\section{Scale of generalized Wolff potential conditions}
All the indexations throughout this section are restricted to the subcollection
$$
\cq:=\{Q\in \cd : \lambda_Q>0, \sigma(Q)>0, \text{ and }  \omega(Q)>0\}
$$
of dyadic cubes, and hence no division by zero occurs.  We abbreviate the indexation  \quotes{$Q\in \cq$} as \quotes{$Q$}. 
\label{sec_generalized_wolff}
\subsection{Sufficiency for large parameters and related conditions}
Applying characterizations by auxiliary coefficients through constructing auxiliary families by hand, we prove the following proposition:
\begin{proposition}[Sufficient integral conditions]Let $0<q< 1< p<\infty$. Let $\sigma$ and $\omega$ be locally finite Borel measures. Let $\{\lambda_Q\}_{Q\in\cq}$ be a family of non-negative reals associated with  the operator $T_\lambda(\cdotroomy \sigma)$. Then the two-weight norm inequality \eqref{eq_norminequality} holds if any one of the following integral conditions is satisfied:
\begin{enumerate}[label=(\roman*)]
\item (Wolff potential condition) We have
\begin{equation}\label{eq_wolffpotentialcondition}
\int \big( \sum_{Q\in \cq}  \lambda_Q \Big(\frac{\omega(Q)}{\sigma(Q)} \Lambda_Q \Big)^{p'-1}1_Q\big)^{\frac{(p-1)q}{p-q}} \domega < \infty.
\end{equation}
\item (Variant of the Wolff potential condition) Let $\gamma\in(0,\infty)$. We have
\begin{equation}\label{eq_wolffpotentialcondition_variant}
\int \big( \sum_{Q\in \cq} \lambda_Q \Big((\Lambda_{\gamma-1,Q})^{1-\gamma} \big(\sup_{R\supseteq Q} \frac{\omega(R)}{\sigma(R)} (\Lambda_{\gamma,R})^{\gamma}\big)\Big)^{p'-1} 1_Q\big)^{\frac{(p-1)q}{p-q}} \domega < \infty.
\end{equation}

\end{enumerate}

\end{proposition}
\begin{proof} First, we check the sufficiency of  condition \eqref{eq_wolffpotentialcondition_variant}. By the characterization by auxiliary coefficients (Theorem \ref{theorem_maintheorem}), it suffices to construct a family $\{d_Q\}$ that satisfies the conditions:
\begin{subequations}
\begin{align}
&\label{condition___da}\sup_Q \frac{1}{\sigma(Q)} \sum_{R\subseteq Q} \lambda_R d_R^{-1} \omega(Q)   < \infty\\
&\label{condition___db} \int \big( \sum_Q\lambda_Q d_Q^{p'-1} 1_Q \big)^{\frac{(p-1)q}{p-q}} < \infty 
\end{align}
\end{subequations}
We choose
$$
d_Q:= (\Lambda_{\gamma-1,Q})^{1-\gamma} \sup_{R\supseteq Q}\frac{\omega(R)}{\sigma(R)} \Lambda_{\gamma,R}^\gamma,
$$
so that condition $\eqref{condition___db}$ becomes the assumed condition \eqref{eq_wolffpotentialcondition_variant}. It remains to check condition \eqref{condition___da} as follows. By writing out the expression, we have
$$
d_Q^{-1}= \frac{\frac{1}{\omega(Q)} \int \big( \sum_{R\subseteq Q} \lambda_R 1_R\big)^{\gamma-1} \domega }{\sup_{R\supseteq Q} \frac{1}{\sigma(R) }\int \big( \sum_{S\subseteq R} \lambda_S 1_S \big)^{\gamma} \domega}.
$$
By omitting all but one cube from the supremum, and by summation by parts (the comparison \eqref{equation:summation_by_parts_inside_integral}), we have
\begin{equation*}
\begin{split}
&\sum_{Q\subseteq P} \lambda_Q d_Q^{-1} \omega(Q)\leq \sigma(P) \frac{\int \sum_{Q\subseteq P} \lambda_Q \big( \sum_{R\subseteq Q} \lambda_R 1_R\big)^{\gamma-1} \domega }{\int \big( \sum_{S\subseteq P} \lambda_S 1_S \big)^{\gamma} \domega}\leq \frac{1}{\gamma} \sigma(P).
\end{split}
\end{equation*}

Next, we prove the sufficiency of condition \eqref{eq_wolffpotentialcondition}.  By the characterization via auxiliary coefficients (Theorem \ref{theorem_factorization} combined with Lemma \ref{lemma:subconditions_trivial_factorization}), it suffices to construct families $\{a_Q\}$ and $\{c_Q\}$ that satisfy the conditions:

\begin{subequations}
\begin{align}
\label{condition___ca}& \sup_Q \frac{c_Q^{-1}}{a_Q^{-1}\sigma(Q)} \sum_{R\subseteq Q} \lambda_R a_R^{-1} \omega(R)< \infty 
 \\
\label{condition___cb}&\sum_Q \lambda_Q a_Q^{-p'} c_Q^{p'-1} \omega(Q) < \infty \\
\label{condition___cc}&\int \big( \sup_Q a_Q 1_Q \big)^{\frac{q}{1-q}} \domega < \infty.
\end{align}
\end{subequations}
We choose
$$
a_Q:= \big( \sum_{R\supseteq Q} \lambda_R \Big( \frac{\omega(R)}{\sigma(R)} \Big)^{p'-1} \Lambda_R^{p'-1} \Big)^{\frac{(p-1)(1-q)}{p-q}},
$$
so that condition \eqref{condition___cc} becomes the assumed condition \eqref{eq_wolffpotentialcondition}. Since $a_Q\geq a_R$ whenever $Q\subseteq R$, for condition \eqref{condition___ca} it suffices that
$$
\sum_{R\subseteq Q} \lambda_R \omega(R) \leq c_Q \sigma(Q) \quad\text{for every $Q$},
$$
which is satisfied by choosing $$c_Q:= \frac{1}{\sigma(Q)}\sum_{R\subseteq Q} \lambda_R \omega(R)=:\frac{\omega(Q)}{\sigma(Q)} \Lambda_Q.$$
Under these choices, condition \eqref{condition___cb} is written out as
$$
\sum_Q \lambda_Q \Big(\frac{\omega(Q)}{\sigma(Q)}\Big)^{p'-1} \Lambda_Q^{p'-1} \omega(Q) \big(\sum_{R\supseteq Q} \lambda_R \Big(\frac{\omega(R)}{\sigma(R)}\Big)^{p'-1} \Lambda_R^{p'-1}\big)^{\frac{(p-1)q}{p-q}-1}< \infty,
$$
which, by summation by parts (comparison \eqref{equation:summation_by_parts_inside_integral}), is comparable to the assumed condition \eqref{eq_wolffpotentialcondition}. The proof is complete.

\end{proof}
\subsection{A counterexample to sufficiency for small parameters}\label{section:unsufficiency}

\begin{proposition} Let $0<q<1<p<\infty$. Let $\gamma \in (0,q)$. Then there exist coefficients $\lambda$, and measures $\sigma$ and $\omega$ such that the necessary condition
\begin{equation}\label{eq:condition_necessary}
\sup_Q \frac{1}{\sigma(Q)^{\frac{1}{p}}} \big(\int \big(\sum_{R\subseteq Q} \lambda_R 1_R\big)^q \domega \big)^{\frac{1}{q}}< \infty
\end{equation}
fails, but yet the condition
\begin{equation}\label{eq:condition_unsufficient}
\int \big( \sum_Q \lambda_Q \Big(\frac{\omega(Q)}{\sigma(Q)}\Big)^{p'-1} \Lambda_{\gamma,Q}^{p'-1} 1_Q \big)^{\frac{(p-1)q}{p-q}}\domega<\infty.
\end{equation}
holds (and thereby this condition is not sufficient).
\end{proposition}
\begin{proof}
Let $P_0\supseteq P_1 \supseteq \cdots$ be a decreasing sequence of nested dyadic cubes. Define the dyadic annuli $E_j$ by $E_j:=P_j\setminus P_{j+1}$. We construct the counterexample by choosing the operator's coefficients $\{\lambda_{P_j}\}$, the $\sigma$-measures $\{\sigma(P_j)\}$ of the cubes $P_j$, and the $\omega$-measures $\{\omega(E_j)\}$ of the annuli $E_j$ such that the quantity in condition  \eqref{eq:condition_necessary} is infinite but yet the quantity in condition \eqref{eq:condition_unsufficient} is finite. Note that the only constraints on the choice of these sequences is that they are non-negative and that $\sigma(P_0)\geq  \sigma(P_{1})\geq \cdots$. 

First, we prepare for the computations. We note that, through integration by parts, for all exponents $\alpha$ and $\delta$ with $\delta\neq 0$ and for all integration limits $c$ and $d$, we have 
\begin{equation*}
\begin{split}
&\int_c^d (t+1)^{\alpha-1 } \log (t+1)^\delta \big( \alpha+\frac{\delta}{\log (t+1)} \big) \dt\\
&= (d+1)^{\alpha} \log (d+1)^\delta-(c+1)^{\alpha} \log (c+1)^\delta.
\end{split}
\end{equation*}
Therefore, for all exponents $\alpha, \beta, \delta >0$, we have 
$$
\sum_{j=0}^k j^{\alpha-1} \log (j+2)^{-\delta}\lesssim k^{\alpha} \log (k+2)^{-\delta}
$$
and
$$
\sum_{j=k}^\infty (j+1)^{-\beta-1} \log (j+2)^{-\delta}\lesssim (k+1)^{-\beta} \log (k+2)^{-\delta}
$$
for sufficiently large $k$.

Next, we choose the sequences $\{\lambda_{P_j}\}$, and $\{\omega(E_j)\}$. Let  $\alpha,\beta,\delta>0$ be exponents that we will pick later. We choose
$$
\lambda_{P_j}:= j^{\alpha-1} \log (j+2)^{-\delta} \quad \text{and} \quad \omega(E_j):= (j+1)^{-\beta-1}.
$$

With these choices, we estimate the quantity in  condition \eqref{eq:condition_unsufficient}. Writing out, we have
\begin{equation*}
\begin{split}
&\int \big( \sum_Q \lambda_Q \Big(\frac{\omega(Q)}{\sigma(Q)}\Big)^{p'-1} \Lambda_{\gamma,Q}^{p'-1} 1_Q \big)^{\frac{(p-1)q}{p-q}}\domega\\
&=\sum_{k=0}^\infty \omega(E_k) \Big( \sum_{l=0}^k \lambda_{P_l} \Big( \frac{\omega(P_l)}{\sigma(P_l)}\Big)^{p'-1} \Big(\frac{1}{\omega(P_l)} \sum_{m=l}^\infty \omega(E_m) \big(\sum_{n=l}^m \lambda_{P_n} \big)^\gamma \Big)^{\frac{p'-1}{\gamma}} \Big)^{\frac{(p-1)q}{p-q}}.
\end{split}
\end{equation*}
We start computing the relevant sums appearing in the quantity:
\begin{equation}\label{temp:6}
\begin{split}
\sum_{n=l}^m \lambda_{P_n}&\leq \sum_{n=0}^m \lambda_{P_n} = \sum_{n=0}^m n^{\alpha-1} \log (n+2)^{-\delta} \lesssim m^{\alpha} \log (m+2)^{-\delta}.\\
\sum_{m=l}^\infty \omega(E_m) \big(\sum_{n=l}^m \lambda_{P_n}\big)^\gamma &\leq \sum_{m=l}^\infty m^{-\beta-1 +\alpha\gamma} \log (m+2)^{-\gamma \delta}\\
&\lesssim l^{-\beta+\alpha\gamma} \log (l+2)^{-\gamma \delta}, \text{ assuming $\beta-\alpha\gamma>0$.}\\
\omega(P_l)&=\sum_{j=l}^\infty \omega(E_j)= \sum_{j=l}^\infty (j+1)^{-\beta+1}\eqsim (l+1)^{-\beta}.\\
\Lambda_{\gamma,P_l}&:= \Big(\frac{1}{\omega(P_l)} \sum_{m=l}^\infty \omega(E_m) \big(\sum_{n=l}^m \lambda_{P_n} \big)^\gamma \Big)^{\frac{1}{\gamma}} \lesssim l^{\alpha} \log (l+2)^{-\delta}.\\
\end{split}
\end{equation}
Now, we choose the sequence $\{\sigma(P_j)\}$. Let $\epsilon>0$ be an exponent that we will pick later. We choose 
$$
\sigma(P_j):= \log (j+2)^{-\epsilon}.
$$
With this choice, we continue computing the relevant sums:
\begin{equation}\label{temp:7}
\begin{split}
&\sum_{l=0}^k \lambda_{P_l} \Big( \frac{\omega(P_l)}{\sigma(P_l)}\Big)^{p'-1} \Lambda_{\gamma,P_l}^{p'-1}\\
&\lesssim \sum_{l=0}^k l^{\alpha p' -\beta (p'-1)-1 } \log (l+2)^{- (\delta p'-\epsilon(p'-1))}\\
&\lesssim k^{\alpha p' -\beta (p'-1)} \log(k+2)^{-(\delta p'-\epsilon(p'-1))}, \text{ assuming $\alpha p' -\beta (p'-1)>0$.}\\
&\sum_{k=0}^\infty \omega(E_k) \Big( \sum_{l=0}^k \lambda_{P_l} \Big( \frac{\omega(P_l)}{\sigma(P_l)}\Big)^{p'-1} \Lambda_{\gamma,P_l}^{p'-1}\Big)^{\frac{(p-1)q}{p-q}}\\
&\lesssim \sum_{k=0}^\infty  (k+1)^{-(\beta \frac{p}{p-q}- \alpha \frac{pq}{p-q})-1} \log (k+2)^{-(\delta \frac{pq}{p-q} -\epsilon \frac{q}{p-q})}.
\end{split}
\end{equation}
Thereby, finally, we have obtained the following upper estimate for the quantity in condition \eqref{eq:condition_unsufficient}:
\begin{equation}\label{temp:temp4}
\begin{split}
&\int \big( \sum_Q \lambda_Q \Big(\frac{\omega(Q)}{\sigma(Q)}\Big)^{p'-1} \Lambda_{\gamma,Q}^{p'-1} 1_Q \big)^{\frac{(p-1)q}{p-q}}\domega\\
&\lesssim \sum_{k=0}^\infty  (k+1)^{-(\beta \frac{p}{p-q}- \alpha \frac{pq}{p-q})-1} \log (k+2)^{-(\delta \frac{pq}{p-q} -\epsilon \frac{q}{p-q})}.
\end{split}
\end{equation}

Next, we obtain the following lower estimate for the quantity in  condition \eqref{eq:condition_necessary}:
\begin{equation}\label{temp:temp5}
\begin{split}
&\sup_{Q} \frac{1}{\sigma(Q)^{\frac{q}{p}}} \int \rho_Q^q \domega\geq  \frac{1}{\sigma(P_0)^{q/p}}  \int \rho_{P_0}^q \domega= \frac{1}{\sigma(P_0)^{q/p}}\sum_{k=0}^\infty \omega(E_k) \big(\sum_{l=0}^k \lambda_{P_l}\big)^q\\
&=\frac{1}{\sigma(P_0)^{q/p}}\sum_{k=0}^\infty (k+1)^{-\beta -1} \big( \sum_{l=0}^k l^{\alpha -1 } \log (l+2)^{-\delta }\big)^q\\
&\eqsim \frac{1}{\sigma(P_0)^{q/p}} \sum_{k=0}^\infty (k+1)^{-\beta -1+\alpha q}\log (l+2)^{-q \delta},
\end{split}
\end{equation}

We complete the proof by picking the exponents $\alpha, \beta, \delta,$ and $\epsilon$. First, we pick $\alpha, \beta,$ and $\delta$ such that the necessary condition \eqref{eq:condition_necessary} fails, which is to say that the quantity \eqref{temp:temp5} is infinite. That is obtained by picking 
$$\alpha q=:\beta \quad\text{ and } \quad  q\delta:=1.$$
Next, we choose $\epsilon$ so that condition \eqref{eq:condition_unsufficient} holds, which is to say that the quantity \eqref{temp:temp4} is finite. With the already made choices for $\alpha, \beta, \delta,$ and $\epsilon$, that is obtained by picking any $\epsilon$ such that $\epsilon<1$. Note that these choices for the exponents $\alpha, \beta, \delta,$ and $\epsilon$ satisfy the assumptions appearing in the intermediate computations \eqref{temp:6} and \eqref{temp:7}. The proof is complete.
\end{proof}

\begin{remark} In the endpoint case $p=1$, a similar counterexample yields the following proposition: When the  integrability parameter $\gamma$ is small so that $\gamma \in(0,q)$, the endpoint generalized Wolff potential condition 
$$
\int \big( \sup_Q \Big(\frac{\omega(Q)}{\sigma(Q)}\Big) \Lambda_{\gamma,Q} 1_Q\big)^{\frac{q}{1-q}}\domega<\infty
$$
is in general not sufficient for the endpoint inequality 
$$
\norm{\sum_Q \lambda_Q a_Q 1_Q}_{L^q(\omega)}\lesssim \norm{\sup_Q a_Q 1_Q}_{L^p(\sigma)}\quad \text{for every $\{a_Q\}$}.
$$
\end{remark}

\subsection{Necessity for small parameters }
Fix an integrability parameter $\gamma\in(0,q)$. We recall that $\Lambda_{\gamma,Q}:= \Big(\frac{1}{\omega(Q)}\int \big(\sum_{R\subseteq Q} \lambda_R 1_R \big)^\gamma\domega \Big)^{\frac{1}{\gamma}}$. We prove that the estimate
\begin{equation}\label{eq_necessarywolf}
\Big( \int \Big( \sum_Q \lambda_Q \big( \frac{\omega(Q)}{\sigma(Q)} \big)^{p'-1} \Lambda_{\gamma,Q}^{p'-1} 1_Q \Big)^{\frac{(p-1)q}{p-q}} \domega \Big)^{\frac{p-q}{pq}} \leq C.
\end{equation}
follows from inequality \eqref{eq_norminequality}. 

First, we notice two auxiliary estimates that are necessary for the norm inequality \eqref{eq_norminequality}.  As recorded in Lemma \ref{lemma_immediateconditions}, the following estimates are necessary:
\begin{equation}\label{eq_necessity_equiv}
\norm{\sum_Q \lambda_Q a_Q 1_Q}_{L^q(\omega)} \leq C \norm{\sup_Q a_Q 1_Q}_{L^p(\sigma)},
\end{equation}
and
\begin{equation}\label{eq_necessity_stein}
\norm{\sum_Q \Lambda_{\gamma,Q} a_Q 1_Q}_{L^q(\omega)}\lesssim_{\gamma,q} \norm{\sum_Q a_Q 1_Q}_{L^p(\sigma)}.
\end{equation}

Next, we dualize the claimed estimate \eqref{eq_necessarywolf}, and the auxiliary necessary estimates \eqref{eq_necessity_equiv} and \eqref{eq_necessity_stein}.
By duality in terms of the discrete Littlewood--Paley norms (Proposition \ref{proposition_normduality}),  estimate \eqref{eq_necessarywolf} is equivalent to the estimate
\begin{equation}\label{eq_necessity_conjeture_dual}
\sum_Q \lambda_Q^{\frac{(p-1)q}{p-q}} \Big(\frac{\omega(Q)}{\sigma(Q)}\Big)^{\frac{p}{p-q}} \Lambda_{\gamma,Q}^{\frac{q}{p-q}} b_Q^{\frac{p(1-q)}{p-q}} \sigma(Q)\leq C^{\frac{pq}{p-q}} \norm{b}_{f^{\infty,1}(\omega)}^{\frac{p(1-q)}{p-q}},
\end{equation}
 estimate \eqref{eq_necessity_equiv} to the estimate
\begin{equation}\label{eq_necessity_equiv_dual}
\Big( \int \big( \sum_Q \lambda_Q^q \frac{\omega(Q)}{\sigma(Q)} b_Q^{1-q} 1_Q \big)^{\frac{p}{p-q}} \dsigma \Big)^{\frac{p-q}{p}}\leq C^q \norm{b}_{f^{\infty,1}(\omega)}^{(1-q)},
\end{equation}
and estimate \eqref{eq_necessity_stein} to the estimate
\begin{equation}\label{eq_necessity_stein_dual}
\Big( \int \big( \sum_Q \Lambda_{\gamma,Q}^{\frac{q}{1-q}} \Big(\frac{\omega(Q)}{\sigma(Q)}\Big)^{\frac{1}{1-q}}  b_Q 1_Q \big)^{\frac{p(1-q)}{p-q}} \dsigma \Big)^{\frac{p-q}{p}}\leq C^q \norm{b}_{f^{\infty,1}(\omega)}^{(1-q)}.
\end{equation}

 Finally, we observe that the claimed estimate follows from the auxiliary necessary estimates by H\"older's inequality. 
We define the H\"older exponents $r,s\in(1,\infty)$ by 
\begin{eqnarray*}
  \left\{
  \begin{aligned}
\frac{r}{s}&:=\frac{p}{p-q}\\
\frac{r'}{s'}&:= \frac{p(1-q)}{p-q}
  \end{aligned}
  \right.
\end{eqnarray*}
and the families $\{c_Q\}$ and $\{d_Q\}$ by
\begin{eqnarray*}
  \left\{
  \begin{aligned}
c_Q^s&:=\lambda_Q^q \frac{\omega(Q)}{\sigma(Q)} b_Q^{1-q}  \\
d_Q^{s'}&:= \Lambda_{\gamma,Q}^{\frac{q}{1-q}} \Big(\frac{\omega(Q)}{\sigma(Q)}\Big)^{\frac{1}{1-q}}  b_Q.
  \end{aligned}
  \right.
\end{eqnarray*}
In terms of these, the claimed estimate \eqref{eq_necessity_conjeture_dual} is rewritten as
$$
\norm{c\cdot d}_{f^{1,1}(\sigma)}\leq C^{\frac{pq}{p-q}} \norm{b}_{f^{\infty,1}(\omega)}^{\frac{p(1-q)}{p-q}},,
$$
and the auxiliary necessary estimates \eqref{eq_necessity_equiv} and \eqref{eq_necessity_stein} as
$$
\norm{c}_{f^{r,s}(\sigma)}\leq C^{\frac{(p-1)q}{p-q}}  \norm{b}_{f^{\infty,1}(\omega)}^{(1-q)\frac{p-1}{p-q}} \quad \text{ and } \quad 
\norm{d}_{f^{r',s'}(\sigma)}\leq C^{\frac{q}{p-q}}  \norm{b}_{f^{\infty,1}(\omega)}^{\frac{1-q}{p-q}}.
$$
Using  H\"older's inequality $\norm{c\cdot d}_{f^{1,1}(\sigma)}\leq \norm{c}_{f^{r,s}(\sigma)}\norm{d}_{f^{r',s'}(\sigma)}$  completes the proof.

\subsection{A counterexample to necessity for large parameters}

\begin{proposition}\label{prop_unnecessityforlargeparameters} Let $0<q<1< p<\infty$.  Then there exist measures $\sigma, \omega$ and coefficients $\lambda$ such that that the condition
\begin{equation}\label{eq_ce_sufficientcondition}
\int_Q \big( \sum_Q \lambda_Q^q \frac{\omega(Q)}{\sigma(Q)} 1_Q \big)^{\frac{p}{p-q}} \dsigma < \infty
\end{equation}
holds, but yet the condition
\begin{equation}\label{cf_unnnecessarycondition}
\int \sup_Q \lambda_Q^{\frac{(p-1)q}{p-q}} \Big(\frac{\omega(Q)}{\sigma(Q)}\Big)^{\frac{q}{p-q}} \Lambda_{q,Q}^\frac{q}{p-q} 1_Q \domega < \infty 
\end{equation}
fails.
\end{proposition}
We note that, since condition \eqref{eq_ce_sufficientcondition} is sufficient for the norm inequality \eqref{eq_norminequality} (by Lemma \ref{lemma_immediateconditions}), the proposition implies that condition \eqref{cf_unnnecessarycondition} and also the stronger condition 
$$
\int \big( \sum_Q \lambda_Q \Big(\frac{\omega(Q)}{\sigma(Q)}\Big)^{p'-1} \Lambda_{q,Q}^{p'-1} 1_Q \big)^{\frac{(p-1)q}{p-q}} \domega < \infty 
$$
are both not necessary for the weighted norm inequality.

\begin{proof}[Proof of Proposition \ref{prop_unnecessityforlargeparameters}]
Let $P_0\subseteq P_1\subseteq \cdots$ be an increasing sequence of cubes. We define the dyadic annuli $E_j$ by  $E_0:=P_0$ and $E_j:=P_j\setminus P_{j-1}$ for $j\geq 1$. Let $\sigma$ be a measure that is supported and non-vanishing on the cube $P_0$. (This is the only requirement for the measure $\sigma$ in this counterexample.) Let $\{\lambda_Q\}$ be coefficients that are non-vanishing only for the cubes $P_j$.

By decomposing the domain of integration into the dyadic annuli, by writing out the nested summation, and by pulling out the constant measure $\sigma(P_j)=\sigma(P_0)$, we write out the left-hand side of condition \eqref{eq_ce_sufficientcondition} as follows:
$$
\text{LHS of \eqref{eq_ce_sufficientcondition}} = \sigma(P_0)^{- \frac{q}{p-q}}\sum_{j=0}^\infty \lambda_{P_j}^q \omega(P_j). 
$$

We notice that the localized sum $\rho_Q:= \sum_{R\subseteq Q} \lambda_R 1_R$  is increasing: $\rho_Q\leq \rho_R$ whenever $Q\leq R$. Thus, $\int \rho_{P_j}^q \domega \geq \int \rho_{P_0}^q \domega=\lambda_{P_0}^q \omega(P_0)$. By using this monotonicity, and by omitting all but one cube in the supremum, we estimate the left-hand side of condition \eqref{cf_unnnecessarycondition} as follows:
\begin{equation*}
\begin{split}
&\text{LHS of \eqref{cf_unnnecessarycondition}}\\
&= \sigma(P_0)^{-\frac{q}{p-q}}\sum_{j=0}^\infty \omega(E_j) \sup_{k\geq j} \lambda_{P_k}^{\frac{(p-1)q}{p-q}} \omega(P_k)^{\frac{q}{p-q}} \Big(\frac{1}{\omega(P_k)}\int \big( \rho_{P_k} \big)^q \domega \big)^{\frac{1}{q}} \Big)^{\frac{q}{p-q}}\\
&\geq (\sigma(P_0)^{-1} \lambda_{P_0} \omega(P_0))^{\frac{q}{p-q}}  \sum_{j=0}^\infty \frac{\omega(E_j)}{\omega(P_j)} \big( \lambda_{P_j}^q \omega(P_j) \big)^{\frac{p-1}{p-q}}.
\end{split}
\end{equation*}
Note that, in our situation, the measure $\omega$ on $\br^d$ is determined by choosing the measures $\omega(E_j)$ of the sets $E_j$. Thus, to prove the proposition, it suffices to find sequences $\{\lambda_{P_J}\}$ and $\{\omega(E_j)\}$ that satisfy the requirements:
\begin{equation}
\label{eq_ce_requirements}
\sum_{j=0}^\infty \lambda_{P_j}^q \omega(P_j) < \infty \quad \text{but} \quad \sum_{j=0}^\infty \frac{\omega(E_j)}{\omega(P_j)} \big( \lambda_{P_j}^q \omega(P_j) \big)^{\alpha}=\infty,
\end{equation}
where the exponent $\alpha$ is defined by $\alpha:=\alpha_{p,q}:=\frac{p-1}{p-q} \in(0,1)$. 

To further simplify the search for such sequences, we observe that for every pair of sequences $\{a_j\}$ and $\{b_j\}$ such that $a_0\geq 1$ and $b_j$ is decreasing, we have
\begin{equation}\label{equation:lowerestimate}
\sum_{j=0}^\infty \frac{a_{j+1}}{\sum_{k=0}^j a_k} b_j^\alpha\geq \lim_{j\to \infty } \big(\log \sum_{k=0}^{j} a_k \big) b_j^\alpha - (\log a_0)b_0,
\end{equation}
which follows from the mean value theorem and a telescoping summation.

We note that $\omega(P_j)=\sum_{k=0}^j \omega(E_j)$. From applying the lower estimate \eqref{equation:lowerestimate} to the right-hand condition of  conditions \eqref{eq_ce_requirements}, it follows that it suffices to find sequences $\{\lambda_{P_J}\}$ and $\{\omega(E_j)\}$ that satisfy the requirements:
$$
\sum_{j=0}^\infty \lambda_{P_j}^q \omega(P_j) < \infty \quad \text{but} \quad \lim_{j\to \infty} \log \omega(P_{j}) \big(\lambda_{P_j}^q \omega(P_j) \big) ^{\alpha}
=\infty,
$$
together with the additional requirements that $\omega(P_0)=:\omega(E_0)\geq 1$, the product sequence $\lambda_{P_j}^q \omega(P_j)$ is decreasing, and $\inf_{j\geq 0} \frac{\omega(E_j)}{\omega(E_{j+1})}>0$. We find such sequences, for example, by picking any $\beta\in(1,\infty)$ such that $\alpha\beta \in (0,1)$, which is possible because $\alpha\in(0,1)$, and setting
$$\omega(E_j):=2^{j+1} \quad \text{and} \quad \lambda_{P_j}^q \omega(P_j):=\frac{1}{(j+1)^\beta}.$$ 
\end{proof}

\appendix
\section{Characterizations by means of auxiliary coefficients}\label{appendix_list_conditions}

In this section we summarize various equivalent conditions in terms of  the families of nonnegative  reals 
$a_Q$, $b_Q$, etc., 
used in the main body of the paper. The appropriate powers of the  quantities appearing in the conditions are comparable, with constants that depend only on the exponents $p$ and $q$. These  powers can be determined by matching the homogeneity (with respect to the scaling of the measures $\sigma$ and $\omega$ and the fixed operator's coefficients $\lambda_Q$) in the corresponding estimates.

The two-weight norm inequality \eqref{eq_norminequality} holds for $0<q<1<p<\infty$ if and only if any one (and hence, by equivalence, all) of the following conditions hold:
\begin{enumerate}[label=(\roman*)]
\item There exists a family $\{a_Q\}$ that satisfies the pair of conditions:
\begin{subequations}
\begin{align*}
&\conditionaa,   \\
&\conditionab .
\end{align*}
\end{subequations}

\item There exist families  $\{a_Q\}$  and $\{b_Q\}$ that satisfy the triple of conditions:

\begin{subequations}
\begin{align*}
&\conditionba,  \\
&\conditionbb,\\
&\conditionbc .
\end{align*}
\end{subequations}

\item There exists a family $\{a_Q\}$  that satisfies the pair of conditions:

\begin{subequations}
\begin{align*}
&\conditionda,\\
&\conditiondb.
\end{align*}
\end{subequations}

\item There exist families  $\{a_Q\}$ and $\{b_Q\}$ that satisfy the triple of conditions:
\begin{subequations}
\begin{align*}
&\conditionea,\\
&\conditioneb,\\
&\conditionec.
\end{align*}
\end{subequations}

\item There exist families  $\{a_Q\}$  and $\{b_Q\}$ that satisfy the triple of conditions:

\begin{subequations}
\begin{align*}
&\conditionca,\\
&\conditioncb\\
&\conditioncc.
\end{align*}
\end{subequations}
\end{enumerate}


\begin{thebibliography}{10}

  \bibitem{adams1996}   
  D.~R. Adams and L.~I. Hedberg.
   \newblock Function spaces and potential theory. 
  \newblock {\em Grundlehren der math. Wissenschaften}, 314. Berlin--Heidelberg--New York, Springer, 1996.
  
  \bibitem{cao2017}
Dat Cao and Igor E. Verbitsky. 
\newblock  Nonlinear elliptic equations and intrinsic potentials of Wolff type. 
\newblock {\em J. Funct. Anal.}, 272(1):112--165, 2017. 

\bibitem{cascante2017}
Carme Cascante and Joaquin~M. Ortega.
\newblock Two-weight norm inequalities for vector-valued operators.
\newblock {\em Mathematika}, 63(1):72--91, 2017.

\bibitem{cascante2004}
Carme Cascante, Joaquin~M. Ortega, and Igor~E. Verbitsky.
\newblock Nonlinear potentials and two weight trace inequalities for general
  dyadic and radial kernels.
\newblock {\em Indiana Univ. Math. J.}, 53(3):845--882, 2004.

\bibitem{cascante2006}
Carme Cascante, Joaquin~M. Ortega, and Igor~E. Verbitsky.
\newblock On {$L^p$}-{$L^q$} trace inequalities.
\newblock {\em J. London Math. Soc. (2)}, 74(2):497--511, 2006.

\bibitem{cohn2000}
W.~S. Cohn and I.~E. Verbitsky.
\newblock Factorization of tent spaces and {H}ankel operators.
\newblock {\em J. Funct. Anal.}, 175(2):308--329, 2000.


\bibitem{frazier1990}
Michael Frazier and Bj\"{o}rn Jawerth.  
\newblock A discrete transform and decompositions of distribution spaces.  
\newblock {\em J. Funct. Anal.},  93(1):34--170, 1990.


\bibitem{hanninen2017}
Timo~S. H\"anninen.
\newblock Two-weight inequality for operator-valued positive dyadic operators
  by parallel stopping cubes.
\newblock {\em Israel J. Math.}, 219(1):71--114, 2017.

\bibitem{hanninen2016}
Timo~S. H\"anninen, Tuomas~P. Hyt\"onen, and Kangwei Li.
\newblock Two-weight {$L^p$}-{$L^q$} bounds for positive dyadic operators:
  unified approach to {$p\leq q$} and {$p>q$}.
\newblock {\em Potential Anal.}, 45(3):579--608, 2016.

\bibitem{hytonen2014}
Tuomas~P. Hyt\"onen.
\newblock The {$A_2$} theorem: remarks and complements.
\newblock In {\em Harmonic analysis and partial differential equations}, volume
  612 of {\em Contemp. Math.}, pages 91--106. Amer. Math. Soc., Providence, RI,
  2014.
\newblock arXiv:1212.3840 [math.CA].

\bibitem{kuusi2014} T.~Kuusi and G. Mingione.
\newblock 
\newblock Guide to nonlinear potential estimates.  
\newblock {\em Bull. Math. Sci.}, 4, 1--82, 2014. 

\bibitem{lacey2009}
Michael~T. Lacey, Eric~T. Sawyer, and Ignacio Uriarte-Tuero.
\newblock Two weight inequalities for discrete positive operators.
\newblock 2009.
\newblock arXiv:0911.3437 [math.CA].

\bibitem{lai2015}
Jingguo Lai.
\newblock A new two weight estimates for a vector-valued positive operator.
\newblock 2015.
\newblock Preprint. arXiv:1503.06778 [math.CA].

\bibitem{maurey1974}
Bernard Maurey.
\newblock Th\'eor\`emes de factorisation pour les op\'erateurs lin\'eaires
  \`a valeurs dans les espaces {$L^{p}$}.
\newblock Soci\'et\'e Math\'ematique de France, Paris. 
\newblock {\em Ast\'erisque}, 11, 1974.

\bibitem{mazya2011}
Vladimir G. Maz'ya.
 \newblock Sobolev spaces, with applications to elliptic partial differential equations, 2nd augmented ed.
  \newblock   {\em Grundlehren der math. Wissenschaften}, 342. Springer, Berlin, 2011.

\bibitem{nazarov1999}
F.~Nazarov, S.~Treil, and A.~Volberg.
\newblock The {B}ellman functions and two-weight inequalities for {H}aar
  multipliers.
\newblock {\em J. Amer. Math. Soc.}, 12(4):909--928, 1999.

\bibitem{pisier1986}
Gilles Pisier.
\newblock Factorization of operators through {$L_{p\infty}$} or {$L_{p1}$} and
  noncommutative generalizations.
\newblock {\em Math. Ann.}, 276(1):105--136, 1986.

\bibitem{quinn2017}
Stephen Quinn and Igor E. Verbitsky. 
\newblock A sublinear version of Schur's lemma and elliptic PDE. 
\newblock Preprint. arXiv: 1702.02682 [math.AP]. 

\bibitem{sawyer1992}
E.~Sawyer and R.~L. Wheeden.
\newblock Weighted inequalities for fractional integrals on {E}uclidean and
  homogeneous spaces.
\newblock {\em Amer. J. Math.}, 114(4):813--874, 1992.

\bibitem{scurry2013}
James Scurry.
\newblock A characterization of two-weight inequalities for a vector-valued
  operator.
\newblock 2013.
\newblock Preprint. arXiv:1007.3089 [math.CA].

\bibitem{tanaka2014}
Hitoshi Tanaka.
\newblock A characterization of two-weight trace inequalities for positive
  dyadic operators in the upper triangle case.
\newblock {\em Potential Anal.}, 41(2):487--499, 2014.

\bibitem{tanaka2015}
Hitoshi Tanaka.
\newblock The trilinear embedding theorem.
\newblock {\em Studia Math.}, 227(3):239--248, 2015.

\bibitem{tanaka2016}
Hitoshi Tanaka.
\newblock The {$n$} linear embedding theorem.
\newblock {\em Potential Anal.}, 44(4):793--809, 2016.

\bibitem{treil2015}
Sergei Treil.
\newblock A remark on two weight estimates for positive dyadic operators.
\newblock In {\em Operator-related function theory and time-frequency
  analysis}, volume~9 of {\em Abel Symp.}, pages 185--195. Springer, Cham,
  2015.
\newblock arXiv:1201.1455 [math.CA].

\bibitem{verbitsky1996}
Igor E. Verbitsky.
\newblock Imbedding and multiplier theorems for discrete Littlewood-Paley
spaces.
\newblock {\em Pacific J. Math.}, 176(2):529--556, 1996.

\bibitem{verbitsky2017}
Igor E. Verbitsky. 
\newblock Sublinear equations and Schur's test for integral operators. 
\newblock In {\em 50 Years with Hardy Spaces, a Tribute to Victor Havin}, 
Eds. A. Baranov et al, Birkh\"{a}user, Ser. Operator Theory: Adv. Appl., volume~261, 2017.  


\end{thebibliography}

\end{document}